\documentclass{siamart251216}

\usepackage{amssymb}
\usepackage{bm}
\usepackage{todonotes}
\usepackage[nameinlink]{cleveref}
\usepackage{tikz}
\usepackage{pgfplots}
\pgfplotsset{compat=1.18}
\usepgfplotslibrary{groupplots}

\hypersetup{
    colorlinks,
    linkcolor={red!50!black},
    citecolor={blue!50!black},
    urlcolor={blue!80!black}
}

\definecolor{seaborngreen}{rgb}{0.3333333333333333, 0.6588235294117647, 0.40784313725490196}  
\definecolor{seaborncyan}{rgb}{0.39215686274509803, 0.7098039215686275, 0.803921568627451}  
\definecolor{seabornblue}{rgb}{0.2980392156862745, 0.4470588235294118, 0.6901960784313725}  
\definecolor{seabornpurple}{rgb}{0.5058823529411764, 0.4470588235294118, 0.6980392156862745}  
\definecolor{seabornred}{rgb}{0.7686274509803922, 0.3058823529411765, 0.3215686274509804}  
\definecolor{seabornorange}{rgb}{0.958, 0.476, 0.206}  
\definecolor{seabornsand}{rgb}{0.8, 0.7254901960784313, 0.4549019607843137}  

\DeclareMathOperator{\tr}{tr}

\newcommand{\bbP}{\mathbb{P}}

\newcommand{\bbR}{\mathbb{R}}
\newcommand{\bbS}{\mathbb{S}}

\newcommand{\bbV}{\mathbb{V}}

\newcommand{\bfF}{\mathbf{F}}
\newcommand{\bfG}{\mathbf{G}}

\newcommand{\bfQ}{\mathbf{Q}}
\newcommand{\bfR}{\mathbf{R}}

\newcommand{\bfX}{\mathbf{X}}

\newcommand{\bfn}{\mathbf{n}}

\newcommand{\bft}{\mathbf{t}}

\newcommand{\bfx}{\mathbf{x}}
\newcommand{\bfy}{\mathbf{y}}

\newcommand{\bfLambda}{\bm{\Lambda}}

\newcommand{\bfnu}{\bm{\nu}}

\newcommand{\bfzero}{\mathbf{0}}  

\newcommand{\calE}{\mathcal{E}}

\newcommand{\calJ}{\mathcal{J}}
\newcommand{\calK}{\mathcal{K}}

\newcommand{\calM}{\mathcal{M}}

\newcommand{\calO}{\mathcal{O}}

\newcommand{\rmd}{\mathrm{d}}

\let\div\relax
\DeclareMathOperator{\div}{div}

\newsiamremark{remark}{Remark}
\crefname{remark}{Remark}{Remarks}

\headers{Arbitrary-order SP schemes for geometric flows}{G.~Zhang, B.~D.~Andrews, and P.~E.~Farrell}
\title{Arbitrary-order structure-preserving discretizations for geometric curvature flows\thanks{
\funding{GZ and BDA were supported by the European Union [ERC, GeoFEM, 101164551].
PEF was supported by
the Engineering and Physical Sciences Research Council [grant no.~EP/W026163/1],
the Science and Technology Facilities Council [grant no.~UKRI/ST/B000495/1],
the Donatio Universitatis Carolinae Chair ``Mathematical modelling of multicomponent systems'',
the UKRI Digital Research Infrastructure Programme through the Science and Technology Facilities Council's Computational Science Centre for Research Communities (CoSeC),
and
the National Science Foundation under grant no.~DMS-1929284 while in residence at the Institute for Computational and Experimental Research in Mathematics in Providence, USA.
The authors gratefully acknowledge the hospitality of the Erwin Schrödinger International Institute for Mathematics and Physics, Vienna, where part of this work was carried out.
For the purpose of open access, the authors have applied a CC BY public copyright licence to any author accepted manuscript arising from this submission.}}}
\author{Ganghui Zhang\thanks{Mathematical Institute, University of Oxford, UK (\email{ganghui.zhang@maths.ox.ac.uk})}. \and Boris D.~Andrews\thanks{Mathematical Institute, University of Oxford, UK (\email{boris.andrews@maths.ox.ac.uk}).} \and Patrick E.~Farrell\thanks{Mathematical Institute, University of Oxford, UK and Mathematical Institute, Faculty of Mathematics and Physics, Charles University, Czechia (\email{patrick.farrell@maths.ox.ac.uk}).}}

\begin{document}

\maketitle

\begin{abstract}
    Geometric flows, where an immersed manifold evolves in time according to its own geometry, exhibit important structural properties.
    For example, surface diffusion dissipates surface area while conserving volume;
    it is desirable to preserve these properties on discretization.
    This has motivated a substantial body of research on structure-preserving discretizations for these flows, albeit at low order in time.
    In this work, we present the first discretization of geometric curvature flows (curve shortening/mean curvature flow and curve/surface diffusion) that preserves the evolution of area and volume at arbitrary order in space and time.
    The key idea is to introduce auxiliary variables in a particular way so that the derivation of the area dissipation law can be replicated after discretization with continuous Petrov--Galerkin in time.
    These auxiliary variables are indicated by a general strategy for structure-preservation in time that applies to many other problems.
    The proposed scheme also preserves mesh quality in the same manner as the minimal deformation rate strategy.
    We demonstrate its structure-preserving properties and high-order convergence on several benchmark examples.
\end{abstract}



\section{Introduction}\label{sec:introduction}

Geometric flows, particularly mean curvature flow (MCF) and surface diffusion (SD), play a crucial role in many scientific and engineering applications.
Both flows can be traced back to Mullins' pioneering works on grain boundary evolution in polycrystalline materials~\cite{Mullins1956,Mullins1957}:
MCF describes the motion of idealized grain boundaries driven by curvature~\cite{Mullins1956}, while SD models the evolution of grain boundaries through thermal grooving~\cite{Mullins1957}.
Since then, MCF has been extensively investigated in various applications, including image processing and segmentation~\cite{Alvarez1993,Osher1988}, multiphase flow and interface dynamics~\cite{Bronsard1991}, and surface fairing in computer graphics~\cite{Desbrun1999}.
Likewise, SD has sparked extensive research, including applications in crystal growth~\cite{Gilmer1972,Gomer1990}, the evolution of voids in microelectronic circuits~\cite{Li1999,Barrett2006}, and solid-state dewetting~\cite{Wang2015}.
Such geometric evolution equations on immersed, moving manifolds possess fundamental geometric structures, including conservation laws and dissipation inequalities.
For reliable long-time simulation, it is desirable that numerical methods inherit these properties~\cite{Hairer_Lubich_Wanner_2006}.

We consider geometric flows on closed immersed $d$-manifolds $\calM_t \subset \bbR^{d+1}$ (i.e.~of codimension 1) determined by their mean curvature $\kappa : \calM_t \to \bbR$.
At each time $t$, we parameterize $\calM_t$ by $\bfX_t : \widetilde{\calM} \to \bbR^{d+1}$, where $\widetilde{\calM}$ is a reference $d$-manifold, such that $\calM_t$ is the image $\bfX_t(\widetilde{\calM})$.
In this work we consider two prominent geometric flows:
curve shortening/mean curvature flow and curve/surface diffusion. 
MCF prescribes that $\dot{\bfX}_t \cdot \bfn = \kappa$;
SD prescribes that $\dot{\bfX}_t \cdot \bfn = - \, \Delta_{\calM_t} \kappa$, where $\Delta_{\calM_t}$ denotes the Laplace--Beltrami operator on $\calM_t$, $\bfn : \calM_t \to \bbR^{d+1}$ is the outward-facing unit normal to $\calM_t$, and $\kappa = \Delta_{\calM_t}\bfX \cdot \bfn$ denotes the mean curvature.
On the continuous level, these two geometric PDEs are governed by the motion of $\calM_t$ in the normal direction.

The desirability of structure preservation has led to a rich literature on the discretization of these flows.
The pioneering Dziuk~\cite{Dziuk_1994} and Barrett--Garcke--N\"urnberg (BGN)~\cite{Barrett_Garcke_Nurnberg_2007,Barrett_Garcke_Nurnberg_2008a,Barrett_Garcke_Nurnberg_2008b} schemes guarantee area dissipation for MCF or SD.
For SD, volume-conserving methods were first considered by Jiang \& Li~\cite{Jiang21} in the case of curve diffusion, i.e.~when $d=1$.
Bao \& Zhao~\cite{Bao_Zhao_2021} observed that certain approximations of the normal vector $\bfn$ in the discrete setting ensure volume conservation in SD for $d \in \{1, 2\}$;
this intermediate normal technique has been extended to other related problems in parametric finite element methods~\cite{Bao-Jiang-Li,Bao-Li2023,Bao-Li2024}.
The above methods are restricted to order 1 in both space and time.
At higher order, Garcke, Jiang, Su \& the first author introduced Lagrange multipliers to preserve area dissipation and volume conservation for SD with order 1 in space and 2 in time~\cite{Garcke_et_al_2025a}, and Garcke, N\"urnberg, Praetorius \& the first author developed a structure-preserving scheme with arbitrary order in space but order 1 in time~\cite{Garcke_et_al_2025b}.
Duan, Li and Zhang \cite{Duan2021} also developed arbitrary order area-dissipating schemes for MCF based on an averaged vector field collocation method in time.
We refer readers to~\cite{jiang2024stable,jiang2025predictor,Jiang23,deckelnick2026second,Duan2024} for non-structure-preserving approaches with higher order in space or time, and to~\cite{Li1,Li2,Kovacs-Li-Lubich2019,Bai2023,Bai2024} for the convergence results of discretizations of MCF and SD.
This work proposes the first discretization that is simultaneously arbitrary-order in space, arbitrary-order in time, and structure-preserving.

One of the characteristic challenges of simulating geometric flows is retaining good mesh quality.
On the continuous level the evolution of the manifold is specified entirely by its normal motion, and tangential motions are redundant.
However, when discretized in space, tangential motion changes the mesh, with major impact on the quality of the discretization.
For example, tangential velocity is not controlled in Dziuk's scheme~\cite{Dziuk_1994}, leading to severe mesh distortion after prolonged evolution.
The BGN scheme discussed above~\cite{Barrett_Garcke_Nurnberg_2007,Barrett_Garcke_Nurnberg_2008a,Barrett_Garcke_Nurnberg_2008b} employs a mixed formulation that (when combined with a suitable normal approximation and a lumped quadrature scheme) introduces a tangential constraint that equidistributes mesh vertices for curves.
Minimal deformation rate (MDR) discretizations, proposed by Hu \& Li~\cite{Hu_Li_2022}, choose the tangential motion to minimize a certain deformation rate $\int_{\calM_t}\|\nabla_{\calM_t}\dot{\bfX}_t\|^2$, where $\nabla_{\calM_t}$ denotes the surface gradient on $\calM_t$.
This can be shown to behave similarly to the BGN scheme. The tradeoff is that for small step sizes BGN-based methods can suffer from mesh distortion~\cite{Duan2024} or failure to converge with fully implicit schemes~\cite[Remark 3.1]{barrett2011approximation}, while MDR in its original formulation loses the area stability of the BGN scheme.
Modifications to the MDR scheme to preserve area stability have been proposed by Gao \& Li~\cite{Gao-Li}, and Gao, Garcke, Li \& Tang~\cite{Gao-Garcke-Li-Tang}.
Gao, Li \& Tang~\cite{Gao_Li_Tang_2026} recently proposed a modification to the MDR method, dubbed \emph{dual}-MDR, that introduces an additional auxiliary variable into the discretization.
The dual-MDR scheme was designed to achieve area dissipation while retaining the high mesh quality of MDR within a semi-implicit discretization of order 1 in time. 

The key idea of the discretization we propose is the introduction of auxiliary variables into the MDR scheme in a manner that preserves the dissipation of area while inheriting the mesh stability of MDR.
Area dissipation can be derived for both MCF and SD by testing the equations with the mean curvature $\kappa$;
by introducing a suitable auxiliary variable approximating $\kappa$, the same variational argument can be applied at the semi-discrete level, and at the discrete level under a continuous Petrov--Galerkin (CPG) discretization in time.
This idea of systematically introducing auxiliary variables in CPG discretizations to preserve conservation and dissipation structures is inspired by the general framework proposed by the last two authors~\cite{Andrews_Farrell_2025b}.
On the semi-discrete level, the auxiliary variables indicated by our framework coincide with the dual-MDR formulation discussed above~\cite{Gao_Li_Tang_2026}, but differ at the discrete level.
Our proposed scheme, defined for arbitrary order in time, thus preserves the dissipation of area for both MCF and SD.
For conservation of volume in SD, the relevant test function is the constant function $1$;
this is in the test set used to construct the CPG discretization, and so volume conservation is inherited by CPG automatically.
We anticipate that the same strategy will extend to other geometric flows of interest, such as hyperbolic MCF and Willmore flow.
We observe also that applying this auxiliary variable framework without the tangential motion specified by MDR re-derives the BGN semi-discretization;
however, without the mesh stability provided by MDR, the Newton iterates in an implicit CPG discretization of the BGN scheme fail to converge.

The conservative and dissipative properties of CPG integrators, i.e.~their ability to replicate variational arguments on the discrete level, have been observed on many different problems for decades~\cite{French_Schaeffer_1990,Betsch_Steinmann_2000a,Betsch_Steinmann_2000b,Egger_Habrich_Shashkov_2021}.
Within ODEs, the systematic use of auxiliary variables to facilitate these variational arguments in more general settings can be connected to discrete gradients, originally proposed by McLachlan, Quispel \& Robidoux ~\cite{McLachlan_Quispel_Robidoux_1999}, and later extended to arbitrary order in time by Hairer, Cohen \& Lubich~\cite{Cohen_Hairer_2011,Hairer_Lubich_2014} and to multiple invariants by the last two authors~\cite{Andrews_Farrell_2025a}.
We also refer the reader to \cite{Akrivis_Li_Tang_Zhang_2025} for recently proposed high-order space--time structure-preserving methods for the nonlinear Schr\"odinger equation.

The manuscript is organized as follows.
In \Cref{sec:notation} we fix notation, and introduce the geometric flows we consider.
\Cref{sec:stabilization} is devoted to the introduction of the MDR and dual-MDR formulations in the semi-discrete setting.
In \Cref{sec:time_discrete}, we apply a continuous Petrov--Galerkin discretization to obtain our full structure-preserving scheme, and prove stability in area and volume.
In \Cref{sec:numerics} we provide numerical results to verify these properties, and demonstrate the higher-order convergence rates in space and time.
\Cref{sec:conclusion} concludes with discussion and future work.

\section{Notation and preliminaries}\label{sec:notation}

For convenience, we henceforth omit the explicit statement of time dependence through the subscript $*_t$ for $\bfX$ and $\calM$ when it is clear from context. 

To construct our scheme, we approximate scalar quantities such as $\kappa$ in a finite-dimensional space $\bbV$ on $\widetilde{\calM}$, with $1 \in \bbV$;
one would then evaluate e.g.~$\kappa$ at $\bfx \in \calM$ by pullback as $\kappa(\bfX^{-1}(\bfx))$.
We approximate vector quantities such as $\bfX$ in $\bbV^d$.
In our numerical results in \Cref{sec:numerics}, we choose $\bbV$ to be the degree-$k$ Lagrange finite element space of piecewise polynomials over the fixed reference geometry $\widetilde{\calM}$~\cite{Bai2023,Bai2024,Garcke_et_al_2025b}.

Denote by $(\cdot, \cdot)_\calM$ the $L^2(\calM)$ inner product, and by $\|\cdot\|_\calM$ the $L^2(\calM)$ norm.
Note that this is dependent on $\calM$, and thus dependent on the state $\bfX$, which is itself non-constant in time.
Since all quadrature is performed by pullback to the constant reference manifold $\widetilde{\calM}$, the corresponding Jacobian must be included in the measure.

When considering time discretizations, let $0 = t_0 < t_1 < \cdots < t_N = T$, and define $T_n := (t_n, t_{n+1})$.
We denote by $\bbP_s(T_n; \bbV)$ the space of degree-$s$ polynomials from $T_n$ to $\bbV$.

\subsection{Mean curvature flow (MCF)}

The first geometric PDE we consider is mean curvature flow,
\begin{equation}\label{eq:mcf}
    \dot{\bfX} \cdot \bfn  =  \kappa,
\end{equation}
where (with a slight abuse of notation) we identify $\bfX = \mathrm{id}$ with a function on $\calM$.
In the continuous setting, the evolution of the surface area $S$ can be found as
\begin{multline}\label{eq:mcf_proof}
   \dot{S}
        = \int_\calM \div_\calM \dot{\bfX}
        = (\nabla_\calM \bfX, \nabla_\calM \dot{\bfX})_\calM
        = - \, (\Delta_\calM \bfX, \dot{\bfX})_\calM \\
        = - \, (\Delta_\calM \bfX\cdot \bfn, \dot{\bfX}\cdot \bfn)_\calM
        = - \, \|\kappa\|_\calM^2 \le 0,
\end{multline}
where $\div_\calM \dot{\bfX}$ denotes the tangential divergence of $\dot{\bfX}$ on $\calM$. 
Here the first and second equalities use the Reynolds transport theorem~\cite[Theorem~32, Lemma~9]{BGN20}, the third applies integration by parts, the fourth uses the normality of the velocity $\dot{\bfX}$, and the final equality applies the definition of $\kappa$ and the normal flow from the MCF equation \eqref{eq:mcf}.
This implies that the area $S$ is dissipated for MCF.


\subsection{Surface diffusion (SD)}

Our second system, surface diffusion~\cite{BGN20}, is similar to mean curvature flow \eqref{eq:mcf}, except that it takes the Laplacian of $\kappa$,
\begin{equation}\label{eq:sd}
    \dot{\bfX} \cdot \bfn  =  - \, \Delta_\calM \kappa.
\end{equation}
The evolution of the surface area $S$ can be computed similarly to \eqref{eq:mcf_proof}:
\begin{equation}
    \dot{S}
        = - \, (\Delta_\calM \bfX \cdot \bfn, \dot{\bfX} \cdot \bfn)_\calM
        = (\kappa, \Delta_\calM \kappa)_\calM
        = - \, \|\nabla_\calM \kappa\|_\calM^2 \le 0.
\end{equation}
The volume $V$ of the $(d+1)$-dimensional domain contained within $\calM$ evolves as~\cite[Theorem~33]{BGN20}
\begin{equation}\label{eq:sd_structures}
    \dot{V}
        = \int_\calM \dot{\bfX}\cdot\bfn
        = \int_\calM - \, \Delta_\calM \kappa = 0,
\end{equation}
where the last equality applies integration by parts.
Hence the enclosed volume is conserved, while area is dissipated.
It is advantageous that numerical discretizations of SD preserve the conservation of $V$.

\section{Stable semi-discretization in space}\label{sec:stabilization}

The framework presented in~\cite{Andrews_Farrell_2025b} provides a strategy for modifying a general PDE semi-discretization to preserve conservation laws and dissipation inequalities.
The core principle is the formulation of desirable structures as variational identities.
These variational identities can be preserved on the discrete level by (i)~systematically introducing certain special auxiliary variables, and (ii)~discretizing in time with a continuous Petrov--Galerkin (CPG) scheme.

To fix ideas, consider MCF \eqref{eq:mcf}.
As an initial didactic example of this framework, we begin with an abstract semi-discretization:
find $\bfX \in \bbV^d$ such that
\begin{equation}\label{eq:mcf_bgn_initial}
    (\dot{\bfX} \cdot \bfn, y)_\calM  =  (\kappa(\bfX), y)_\calM,
\end{equation}
for all $y \in \bbV$, where $\kappa(\bfX)$ is a function to be determined of the parametrization function $\bfX$, approximating the true mean curvature $\Delta_\calM \bfX \cdot \bfn$.
Our intention is to apply the framework of~\cite{Andrews_Farrell_2025b} to construct a mixed formulation of the above system that preserves the dissipation of area both (i)~on the semi-discrete level, and (ii)~after discretization in time with CPG.

\begin{remark}[Under-determination]
    The initial semi-discretization \eqref{eq:mcf_bgn_initial} is under-determined: the test space $\bbV$ is smaller in dimension than the solution space $\bbV^d$.
    This is a consequence of the motion $\dot{\bfX}$ \emph{tangential} to the manifold $\calM$ being unspecified.
    This particular semi-discrete problem thus technically lies slightly outside the framework as presented in~\cite{Andrews_Farrell_2025b}.
    Proceeding regardless derives a stable semi-discretization \eqref{eq:mcf_bgn} which loses the under-determined nature;
    after introduction of an auxiliary variable, the solution and test spaces in the final mixed form are identical.
    However this initial ill-posedness has persistent implications for the well-posedness of the CPG discretization in time.
    We discuss this in further detail at the end of this section.
\end{remark}

We seek to choose the discrete mean curvature approximation $\kappa(\bfX)$ so that the dissipation of area $S$ is preserved.
Consider again the first steps of \eqref{eq:mcf_proof} for the evolution of $S$,
\begin{equation}
    \dot{S}
        = \int_\calM \div_\calM \dot{\bfX}
        =  (\nabla_\calM \bfX, \nabla_\calM \dot{\bfX})_\calM
        =  - \, (\Delta_\calM \bfX, \dot{\bfX})_\calM.
\end{equation}
In the continuous setting, we are then able to write $- \, (\Delta_\calM \bfX, \dot{\bfX})_\calM = - \, (\Delta_\calM \bfX \cdot\bfn, \dot{\bfX} \cdot \bfn)_\calM$.
Taking $y = \Delta_\calM \bfX \cdot \bfn$ in \eqref{eq:mcf_bgn_initial} would then assert the dissipation of area;
in general however, $\Delta_\calM \bfX \cdot\bfn \not\in \bbV$, making this invalid as a choice of test function on the discrete level.
The framework of~\cite{Andrews_Farrell_2025b} thus proposes introducing into the semi-discretization an auxiliary variable $\kappa$, approximating $\Delta_\calM\bfX \cdot \bfn$ within the test space $\bbV$, such that it \emph{is} a valid choice of test function.
The definition of $\kappa$ is specified precisely by~\cite{Andrews_Farrell_2025b} as a projection under the dual of the left-hand side of \eqref{eq:mcf_bgn_initial}:
find $\kappa \in \bbV$ such that
\begin{subequations}\label{eq:mcf_bgn}
\begin{equation}
    (\kappa\bfn, \bfLambda)_\calM  =  - \, (\nabla_\calM\bfX, \nabla_\calM\bfLambda)_\calM,
\end{equation}
for all $\bfLambda \in \bbV^d$.
With $\kappa$ so defined, we modify the primal system \eqref{eq:mcf_bgn_initial} to introduce $\kappa$ on the right-hand side:
find $\bfX \in \bbV^d$ such that
\begin{equation}
    (\dot{\bfX} \cdot \bfn, y)_\calM  =  (\kappa, y)_\calM,
\end{equation}
\end{subequations}
for all $y \in \bbV$.
This careful introduction of auxiliary variables allows us to ``pass by'' $\kappa$ when evaluating $\dot{S}$:
\begin{equation}
    \dot{S}
        =  (\nabla_\calM \bfX, \nabla_\calM \dot{\bfX})_\calM
        =  - \, (\kappa\bfn, \dot{\bfX})_\calM
        =  - \, \|\kappa\|_\calM^2  \le  0,
\end{equation}
where in the second equality we consider $\bfLambda = \dot{\bfX}$, and in the third and final equality we consider $y = \kappa$.
Thus, this \emph{mixed} semi-discretization \eqref{eq:mcf_bgn} can be seen to preserve the dissipation of area.

The above semi-discretization clearly aligns with the BGN method~\cite{Barrett_Garcke_Nurnberg_2007,Barrett_Garcke_Nurnberg_2008a,Barrett_Garcke_Nurnberg_2008b}, which is well-known to be area-decreasing.
While this stability result would theoretically extend naturally to CPG discretizations in time, it has been observed that the Newton iterates in nonlinear time discretizations of \eqref{eq:mcf_bgn} struggle to converge~\cite[Remark 3.1]{barrett2011approximation}, with the system becoming increasingly ill-conditioned on smaller timesteps.
Carefully constructed linear discretizations that preserve the dissipation of area are preferred~\cite{BGN20}; however these can still cause severe mesh deformation and fail to simulate the true geometric behavior~\cite{Duan2024}.
This can be attributed to the ill-posedness of the initial under-determined abstract semi-discrete problem \eqref{eq:mcf_bgn_initial}.
One solution is to specify the tangential motion at the initial step, before applying the auxiliary variable framework.

\subsection{Mesh stabilization: MDR and dual-MDR}

The MDR approach~\cite{Hu_Li_2022,Gao_Li_Tang_2026} constrains the tangential motion to that which minimizes a deformation rate $\|\nabla_\calM\dot{\bfX}\|_\calM^2$, by positing, on the continuous level, the existence of a scalar field $p$ such that
\begin{equation}\label{eq:MDR}
    \Delta_\calM\dot{\bfX} = p\bfn.
\end{equation}
MDR presents a general, approachable strategy for improving mesh quality.
It is universal, in so far as it is not restricted to a certain geometric flow such as MCF or SD, and on the continuous level will have no effect on the normal motion of $\dot{\bfX}$, and consequently no effect on the fundamental dynamics of the manifold $\calM$.
We now apply the auxiliary variable framework~\cite{Andrews_Farrell_2025b} to an abstract MDR semi-discretization to preserve the dissipation of area.

Again fixing ideas with MCF \eqref{eq:mcf}, one begins with an abstract MDR semi-discretization:
find $(p, \bfX) \in \bbV \times \bbV^d$ such that
\begin{subequations}\label{eq:mcf_mdr_initial}
\begin{align}
    (\dot{\bfX} \cdot \bfn, y)_\calM  &=  (\kappa(\bfX), y)_\calM,  \\
    (\nabla_\calM\dot{\bfX}, \nabla_\calM\bfQ)_\calM + (p\bfn, \bfQ)_\calM  &=  0,
\end{align}
\end{subequations}
for all $(y, \bfQ) \in \bbV \times \bbV^d$.
Taking $y = \Delta_\calM \bfX \cdot \bfn$, or more specifically $(y, \bfQ) = (\Delta_\calM \bfX \cdot\bfn, \bfzero)$, in \eqref{eq:mcf_mdr_initial} would then theoretically prove area dissipation.
The framework of~\cite{Andrews_Farrell_2025b} thus proposes introducing auxiliary variables $(\kappa, \bfR)$, approximating $(\Delta_\calM\bfX \cdot \bfn, \bfzero)$ within $\bbV \times \bbV^d$.
The definition of $(\kappa, \bfR)$ is specified by~\cite{Andrews_Farrell_2025b}:
find $(\kappa, \bfR) \in \bbV \times \bbV^d$ such that
\begin{subequations}\label{eq:mcf_mdr}
\begin{align}
    (\bfR\cdot\bfn, s)_\calM  &=  0,  \label{eq:mcf_mdr_a}  \\
    (\nabla_\calM\bfR, \nabla_\calM\bfLambda)_\calM + (\kappa\bfn, \bfLambda)_\calM  &=  - \, (\nabla_\calM\bfX, \nabla_\calM\bfLambda)_\calM,  \label{eq:mcf_mdr_b}
\end{align}
for all $(s, \bfLambda) \in \bbV \times \bbV^d$.
We modify the primal MDR system \eqref{eq:mcf_mdr_initial} to introduce $\kappa$:
find $(p, \bfX) \in \bbV \times \bbV^d$ such that
\begin{align}
    (\dot{\bfX} \cdot \bfn, y)_\calM  &=  (\kappa, y)_\calM,  \label{eq:mcf_mdr_c}  \\
    (\nabla_\calM\dot{\bfX}, \nabla_\calM\bfQ)_\calM + (p\bfn, \bfQ)_\calM  &=  0,  \label{eq:mcf_mdr_d}
\end{align}
\end{subequations}
for all $(y, \bfQ) \in \bbV \times \bbV^d$.
We thus evaluate $\dot{S}$ for \eqref{eq:mcf_mdr}:
\begin{multline}\label{eq:mcf_mdr_proof}
    \dot{S}
        =  (\nabla_\calM \bfX, \nabla_\calM \dot{\bfX})_\calM
        =  - \, (\nabla_\calM\bfR, \nabla_\calM\dot{\bfX})_\calM - (\kappa\bfn, \dot{\bfX})_\calM  \\
        =  (p\bfn, \bfR)_\calM - (\kappa\bfn, \dot{\bfX})_\calM
        =  - \, \|\kappa\|_\calM^2  \le  0,
\end{multline}
where in the second equality we consider $\bfLambda = \dot{\bfX}$, in the third we consider $\bfQ = \bfR$, and in the final equality we consider $(y, s) = (\kappa, p)$.
This \emph{mixed MDR} semi-discretization \eqref{eq:mcf_mdr} can thus be seen to preserve the dissipation of area, while inheriting the mesh stability property of MDR.
This stability result then extends naturally to CPG discretizations in time (see \Cref{sec:time_discrete} below).

Note, in the continuous setting the variable $\bfR$ will evaluate to $\bfzero$.
In the discrete setting however, $\bfR \ne \bfzero$ in general, and is required to guarantee area stability.
This approach of systematically introducing auxiliary variables motivates and recovers the recently proposed dual-MDR formulation~\cite{Gao_Li_Tang_2026} for mixed methods within MDR, wherein the above semi-discretization \eqref{eq:mcf_mdr} is proposed.
The authors consider a carefully defined discretization in time of \eqref{eq:mcf_mdr}, similar to implicit Euler, that is (i)~linear, and (ii)~provably dissipative in the area for the parabolic MCF and SD systems~\cite[Theorem~2.2]{Gao_Li_Tang_2026}.
The proof of area stability relies on the inequality,
\begin{equation}
    S(t_{n+1}) - S(t_n)
        \le  (\nabla_{\calM_{t_n}}\bfX_{t_{n+1}}, \nabla_{\calM_{t_n}}[\bfX_{t_{n+1}} - \bfX_{t_n}])_{\calM_{t_n}},
\end{equation}
valid for linear finite elements~\cite[(2.31)]{Barrett_Garcke_Nurnberg_2007},~\cite[(2.21)]{Barrett_Garcke_Nurnberg_2008a}.
The extent to which this inequality is not strict quantifies additional, artificial area dissipation in their proposed scheme.

We now state the dual-MDR semidiscrete systems for MCF~\eqref{eq:mcf} and SD~\eqref{eq:sd}.
In both cases, for brevity their structure-preserving properties are presented and proven after time discretization in \Cref{sec:time_discrete}.

\subsection{Mean curvature flow (MCF)}

As motivated above, the semi-discrete dual-MDR formulation of MCF is as follows:

\begin{definition}[Dual-MDR MCF semi-discretization]
    Find $(\bfX, p, \bfR, \kappa) \in \bbV^d \times \bbV \times \bbV^d \times \bbV$ such that
    \begin{subequations}\label{eq:mcf_semidiscrete}
    \begin{align}
        (\dot{\bfX} \cdot \bfn, y)_\calM  &=  (\kappa, y)_\calM,  \label{eq:mcf_semidiscrete_a}  \\
        (\nabla_\calM\dot{\bfX}, \nabla_\calM\bfQ)_\calM + (p\bfn, \bfQ)_\calM  &=  0,  \label{eq:mcf_semidiscrete_b}  \\
        (\bfR\cdot\bfn, s)_\calM  &=  0,  \label{eq:mcf_semidiscrete_c}  \\
        (\nabla_\calM\bfR, \nabla_\calM\bfLambda)_\calM + (\kappa \bfn, \bfLambda)_\calM  &=  - \, (\nabla_\calM \bfX, \nabla_\calM \bfLambda)_\calM,  \label{eq:mcf_semidiscrete_d}
    \end{align}
    \end{subequations}
    for all $(y, \bfQ, s, \bfLambda) \in \bbV \times \bbV^d \times \bbV \times \bbV^d$.
\end{definition}

In the continuous limit, the above semi-discrete system can be seen to correspond to the strong form equations,
\begin{subequations}\label{eq:mcf_mdr_strong}
\begin{align}
    \dot{\bfX} \cdot \bfn  &=  \kappa,  &
    \bfR\cdot\bfn  &=  0,  \\
    \Delta_\calM \dot{\bfX}  &=  p\bfn,  &
    \kappa \bfn  &=  \Delta_\calM \bfX + \Delta_\calM \bfR.
\end{align}
\end{subequations}

\subsection{Surface diffusion (SD)}

Similar to the above, the semi-discrete dual-MDR formulation of SD is as follows:

\begin{definition}[Dual-MDR SD semi-discretization]
    Find $(\bfX, p, \bfR, \kappa) \in \bbV^d \times \bbV \times \bbV^d \times \bbV$ such that
    \begin{subequations}\label{eq:sd_semidiscrete}
    \begin{align}
        (\dot{\bfX} \cdot \bfn, y)_\calM  &=  (\nabla_\calM \kappa, \nabla_\calM y)_\calM,  \label{eq:sd_semidiscrete_a}  \\
        (\nabla_\calM\dot{\bfX}, \nabla_\calM\bfQ)_\calM + (p\bfn, \bfQ)_\calM  &=  0,  \label{eq:sd_semidiscrete_b}  \\
        (\bfR\cdot\bfn, s)_\calM  &=  0,  \label{eq:sd_semidiscrete_c}  \\
        (\nabla_\calM\bfR, \nabla_\calM\bfLambda)_\calM + (\kappa \bfn, \bfLambda)_\calM  &=  - \, (\nabla_\calM \bfX, \nabla_\calM \bfLambda)_\calM,  \label{eq:sd_semidiscrete_d}
    \end{align}
    \end{subequations}
    for all $(y, \bfQ, s, \bfLambda) \in \bbV \times \bbV^d \times \bbV \times \bbV^d$.
\end{definition}

Again in the limit, the above semi-discrete system corresponds with the strong form equations,
\begin{subequations}
\begin{align}
    \dot{\bfX} \cdot \bfn  &=  - \Delta_\calM \kappa,  &
    \bfR\cdot\bfn  &=  0,  \\
    \Delta_\calM \dot{\bfX}  &=  p\bfn,  &
    \kappa \bfn  &=  \Delta_\calM \bfX + \Delta_\calM \bfR.
\end{align}
\end{subequations}
Note, we need not apply the framework of~\cite{Andrews_Farrell_2025b} again to preserve the conservation of volume for this SD semi-discretization, as it can be seen to hold already without modification by considering $y = 1$.

\section{Stable CPG discretization in space and time}\label{sec:time_discrete}

On each timestep $T_n$, we apply a CPG discretization to each of our semi-discretizations to obtain a fully discrete system.
These CPG discretizations, theoretically possible for any number of stages $s$, are shown to preserve each volume $V$ and area $S$ stability result considered above.

\subsection{Mean curvature flow (MCF)}

\begin{definition}[Dual-MDR MCF CPG discretization]
    For each $n$, find $\bfX \in \bbP_s(T_n; \bbV^d)$ satisfying known initial data at $t_n$, and $(p, \bfR, \kappa) \in \bbP_{s-1}(T_n; \bbV \times \bbV^d \times \bbV)$ such that
    \begin{subequations}\label{eq:mcf_cpg}
    \begin{align}
        \int_{T_n} (\dot{\bfX} \cdot \bfn, y)_\calM  &=  \int_{T_n} (\kappa, y)_\calM,  \label{eq:mcf_cpg_a}  \\
        \int_{T_n} \Big[ (\nabla_\calM\dot{\bfX}, \nabla_\calM\bfQ)_\calM + (p\bfn, \bfQ)_\calM \Big]  &=  0,  \label{eq:mcf_cpg_b}  \\
        \int_{T_n} (\bfR\cdot\bfn, s)_\calM  &=  0,  \label{eq:mcf_cpg_c}  \\
        \int_{T_n} \Big[ (\nabla_\calM\bfR, \nabla_\calM\bfLambda)_\calM + (\kappa \bfn, \bfLambda)_\calM \Big]  &=  - \, \int_{T_n} (\nabla_\calM \bfX, \nabla_\calM \bfLambda)_\calM,  \label{eq:mcf_cpg_d}
    \end{align}
    \end{subequations}
    for all $(y, \bfQ, s, \bfLambda) \in \bbP_{s-1}(T_n; \bbV \times \bbV^d \times \bbV \times \bbV^d)$.
\end{definition}

To prove area dissipation, we first note the following lemma.

\begin{lemma}[Orthogonality of $\dot{\bfX}$ and $\bfR$]\label{lem:orthogonality}
    Any discrete solution to \eqref{eq:mcf_cpg} satisfies, on each slab $T_n$,
    \begin{equation}
        \int_{T_n} (\nabla_\calM\dot{\bfX}, \nabla_\calM\bfR)_\calM = 0.
    \end{equation}
\end{lemma}

\begin{proof}
    Testing the second equation \eqref{eq:mcf_cpg_b} with $\bfQ = \bfR$ yields
    \begin{subequations}
    \begin{equation}
        \int_{T_n} \Big[(\nabla_\calM\dot{\bfX}, \nabla_\calM\bfR)_\calM + (p\bfn, \bfR)_\calM\Big] = 0.
    \end{equation}
    Testing the third equation \eqref{eq:mcf_cpg_c} with $s = p$ yields
    \begin{equation}
        \int_{T_n}(\bfR\cdot\bfn, p)_\calM = 0.
    \end{equation}
    \end{subequations}
    Taking the difference of these identities gives the desired result.
\end{proof}

With this result, we are able to show the proposed MCF discretization \eqref{eq:mcf_cpg} dissipates area.

\begin{theorem}[Stability of MCF discretization]\label{thm:mcf_structures}
    Any discrete solution to \eqref{eq:mcf_cpg} preserves the dissipation of area $S$,
    \begin{equation}
        S(t_{n+1}) - S(t_n)  =  - \int_{T_n}  \|\kappa\|_\calM^2 \le 0.
    \end{equation}
    Moreover, the dissipation rate $- \, \|\kappa\|_\calM^2$ mimics that of the continuous system \eqref{eq:mcf_proof}.
\end{theorem}

\begin{proof}
    Testing with $\bfLambda = \dot{\bfX}$ in \eqref{eq:mcf_cpg_d} and $y = \kappa$ in \eqref{eq:mcf_cpg_a},
    \begin{subequations}
    \begin{align}
        S(t_{n+1}) - S(t_n)
            = \int_{T_n} \dot{S}
            &= \int_{T_n}(\nabla_\calM\bfX, \nabla_\calM\dot{\bfX})_\calM  \label{eq:discrete_reynolds}  \\
            &=-  \, \int_{T_n}(\kappa \bfn, \dot{\bfX})_{\calM}- \int_{T_n} (\nabla_\calM\bfR, \nabla_\calM \dot{\bfX})_\calM  \\
            &=  - \, \int_{T_n} \|\kappa\|_\calM^2 \le  0,
    \end{align}
    \end{subequations}
    where we eliminate $\int_{T_n} (\nabla_\calM\dot{\bfX}, \nabla_\calM\bfR)_\calM$ after the third equality via \Cref{lem:orthogonality}.
    We are able to consider $\bfLambda = \dot{\bfX}$ as $\bfX \in \bbP_s(T_n; \bbV^d)$ implies $\dot{\bfX} \in \bbP_{s-1}(T_n; \bbV^d)$, the test set containing $\bfLambda$.
\end{proof}



A detailed derivation of the surface area evolution identity $\dot{S} = (\nabla_\calM\bfX, \nabla_\calM\dot{\bfX})_\calM$ \eqref{eq:discrete_reynolds} on general parameterized finite-element surfaces is provided in \Cref{app:area_identity_parameterized}.


\subsection{Surface diffusion (SD)}

\begin{definition}[Dual-MDR SD CPG discretization]
    For each $n$, find $\bfX \in \bbP_s(T_n; \bbV^d)$ satisfying known initial data at $t_n$, and $(p, \bfR, \kappa) \in \bbP_{s-1}(T_n; \bbV \times \bbV^d \times \bbV)$ such that
    \begin{subequations}\label{eq:sd_cpg}
    \begin{align}
        \int_{T_n} (\dot{\bfX} \cdot \bfn, y)_\calM  &=  \int_{T_n} (\nabla_\calM \kappa, \nabla_\calM y)_\calM,  \label{eq:sd_cpg_a}  \\
        \int_{T_n} \Big[ (\nabla_\calM\dot{\bfX}, \nabla_\calM\bfQ)_\calM + (p\bfn, \bfQ)_\calM \Big]  &=  0,  \label{eq:sd_cpg_b}  \\
        \int_{T_n} (\bfR\cdot\bfn, s)_\calM  &=  0,  \label{eq:sd_cpg_c}  \\
        \int_{T_n} \Big[ (\nabla_\calM\bfR, \nabla_\calM\bfLambda)_\calM + (\kappa \bfn, \bfLambda)_\calM \Big]  &=  - \, \int_{T_n} (\nabla_\calM \bfX, \nabla_\calM \bfLambda)_\calM,  \label{eq:sd_cpg_d}
    \end{align}
    \end{subequations}
    for all $(y, \bfQ, s, \bfLambda) \in \bbP_{s-1}(T_n; \bbV \times \bbV^d \times \bbV \times \bbV^d)$.
\end{definition}

Since the second \eqref{eq:sd_cpg_b} and third \eqref{eq:sd_cpg_c} equations are unchanged from \eqref{eq:mcf_cpg}, \Cref{lem:orthogonality} carries over identically to \eqref{eq:sd_cpg}.
We show then the proposed SD discretization \eqref{eq:sd_cpg} both conserves volume and dissipates area.

\begin{theorem}[Stability of SD discretization]\label{thm:sd_structures}
    Any discrete solution to \eqref{eq:sd_cpg} preserves the conservation of volume $V$, and dissipation of area $S$,
    \begin{equation}
        V(t_{n+1}) - V(t_n)  =  0,  \qquad
        S(t_{n+1}) - S(t_n)  =  - \int_{T_n} \|\nabla_\calM \kappa\|_\calM^2 \le 0.
    \end{equation}
    Again, the correct area dissipation rate $- \, \|\nabla_\calM \kappa\|_\calM^2$ \eqref{eq:sd_structures} holds.
\end{theorem}

\begin{proof}
    For volume, testing with $y = 1$ in the first equation \eqref{eq:sd_cpg_a} yields
    \begin{subequations}
    \begin{align}
        V(t_{n+1}) - V(t_n)
            =  \int_{T_n} \dot{V}
            &=  \int_{T_n} \int_\calM \dot{\bfX} \cdot \bfn  \label{eq:discrete_reynolds_volume}  \\
            &=  \int_{T_n} (\dot{\bfX} \cdot \bfn, 1)_\calM  \\
            &=  \int_{T_n} (\nabla_\calM\kappa, \nabla_\calM 1)_\calM
            =  0.
    \end{align}
    \end{subequations}
    For area, the strategy is similar to the MCF case (see \Cref{thm:mcf_structures}).
    Testing with $\bfLambda = \dot{\bfX}$ in the last equation \eqref{eq:sd_cpg_d} and $y = \kappa$ in the first \eqref{eq:sd_cpg_a},
    \begin{subequations}
    \begin{align}
        S(t_{n+1}) - S(t_n)
            = \int_{T_n}\dot{S}
            &= \int_{T_n}(\nabla_\calM\bfX, \nabla_\calM\dot{\bfX})_\calM  \\
            &=-  \, \int_{T_n}(\kappa \bfn, \dot{\bfX})_{\calM}- \int_{T_n} (\nabla_\calM\bfR, \nabla_\calM \dot{\bfX})_\calM  \\
            &=  - \, \int_{T_n} \|\nabla_\calM \kappa\|_\calM^2 \le  0.
    \end{align}
    \end{subequations}
    Again, we eliminate $\int_{T_n} (\nabla_\calM\dot{\bfX}, \nabla_\calM\bfR)_\calM$ via \Cref{lem:orthogonality}.
\end{proof}

A detailed derivation of the volume $V$ evolution identity $\dot{V} = \int_\calM \dot{\bfX}\cdot\bfn$ \eqref{eq:discrete_reynolds_volume} on general parameterized finite element surfaces is provided in \Cref{app:volume_identity_parameterized}.

\subsection{Time quadrature and sufficient conditions for stability}\label{sec:time_quadrature}

In theory, CPG demands that the time integrals in each of these schemes \eqref{eq:mcf_cpg} and \eqref{eq:sd_cpg} are evaluated exactly.
In practice, however, quadrature is required;
some terms cannot be integrated exactly.
An analysis of which terms can be underintegrated without affecting structure preservation is given in \cite{Andrews_Farrell_2025b}.
The summary of this analysis is that most terms can be underintegrated (with mild constraints on the quadrature rule, like having positive weights) without affecting structure preservation, but that quadrature errors on the right-hand sides of the auxiliary variable equations (i.e.~\eqref{eq:mcf_cpg_d}, \eqref{eq:sd_cpg_d}) do affect it.
If these terms are not polynomial and thus cannot be integrated exactly, we employ a higher-order quadrature only for these terms so that the resulting errors are on the same order as those necessarily incurred by inexact nonlinear solves.
In the following subsections we specify precisely which terms must be so integrated so as not to affect structure preservation.

\subsubsection{Volume conservation in SD}

It is sufficient for volume conservation in SD that the quadrature approximation to $\int_{T_n}(\dot{\bfX}\cdot\bfn, y)_\calM$ evaluates exactly to $V(t_{n+1}) - V(t_n)$ when $y = 1$.

\begin{theorem}[Quadrature on $\int_{T_n} (\dot{\bfX}\cdot\bfn, y)_\calM$]
    Time quadrature of degree $ds + s - 1$ in time applied to the term $\int_{T_n}(\dot{\bfX}\cdot\bfn, y)_\calM$ is sufficient to guarantee volume conservation in SD \eqref{eq:sd_cpg}.
    Moreover, quadrature of degree $ds + 2s - 2$ (greater than $ds + s - 1$ for $s > 1$) is sufficient to integrate this term exactly.
\end{theorem}

\begin{proof}
    Recall that the inner product in the integrand $(\dot{\bfX}\cdot\bfn, y)_\calM$ is evaluated in practice not as an integral over $\calM$, but as an integral over the reference manifold $\widetilde{\calM}$, after pullback and multiplication by the Jacobian $J$ of the map $\bfX : \widetilde{\calM} \to \calM$ \eqref{eq:jacobian}.
    We show first that, after pullback, this integrand is polynomial in $\bfX$ and $y$.

    Suppose $\widetilde{\calM}$ and (consequently) $\calM$ are parameterized by reference coordinates $(\rho_1, \dots, \rho_d)$.
    The normal $\bfn$ on $\calM$ can be found as $\bfn = \bfnu / \|\bfnu\|$ for
    \begin{equation}\label{eq:nu}
        \bfnu  :=  \bigwedge_{i=1}^d \frac{\partial}{\partial\rho_i} \bfX,
    \end{equation}
    where $\bigwedge$ denotes the exterior product. 
    The Jacobian $J$ can be identified exactly with $\|\bfnu\|$.  
    Thus evaluating $(\dot{\bfX}\cdot\bfn, y)_\calM$ over the reference manifold,
    \begin{equation}\label{eq:use_nu}
        (\dot{\bfX}\cdot\bfn, y)_\calM
            =  (\dot{\bfX}\cdot\bfn\|\bfnu\|, y)_{\widetilde{\calM}}
            =  (\dot{\bfX}\cdot\bfnu, y)_{\widetilde{\calM}}.
    \end{equation}
    Over $\widetilde{\calM}$, the integrand reduces to a polynomial in $\dot{\bfX}$, $\bfnu$ and $y$.
    Since $\bfnu$ is itself a polynomial of degree $d$ in $\bfX$ \eqref{eq:nu}, what remains is an integrand that is ultimately polynomial.

    The order of this integrand in time can be identified as $ds + 2s - 2$\footnote{Two copies of $s-1$ from $\dot{\bfX}$ and $y$, plus $d$ copies of $s$ from $\bfnu  =  \bigwedge_{i=1}^d \frac{\partial}{\partial\rho_i} \bfX$.}, and thus can be integrated exactly with any chosen degree-$(ds + 2s - 2)$ quadrature rule.
    As stated above, to guarantee volume conservation we in fact only require the quadrature on $\int_{T_n}(\dot{\bfX}\cdot\bfn, y)_\calM$ to be exact when $y = 1$, reducing the order of the integrand to $ds + s - 1$.
\end{proof}

\begin{remark}[The intermediate normal technique~\cite{Bao_Zhao_2021}]
    For 1-stage methods ($s = 1$), $ds + 2s - 2 = d$.
    The intermediate normal technique of Bao \& Zhao is recovered in the \emph{curve} diffusion case ($d = 1$) when using a trapezium rule (degree $1$)~\cite[(2.10)]{Bao_Zhao_2021}, and in the \emph{surface} diffusion case ($d = 2$) when applying 3-stage Gauss--Lobatto quadrature (degree $3$)~\cite[(3.12)]{Bao_Zhao_2021}.
    For instance in the latter case of $d = 2$,
    \begin{equation}
        \int_{T_n} (\dot{\bfX} \cdot \bfn, y)_\calM
            =  \int_{T_n} (\dot{\bfX} \cdot \bfnu, y)_{\widetilde{\calM}}
            =  \frac{t_{n+1} - t_n}{6}\!\left[\begin{aligned}
                &(\dot{\bfX} \cdot \bfnu_{t_n}, y)_{\widetilde{\calM}}  \\
                &\qquad+ 4(\dot{\bfX} \cdot \bfnu_{t_{n+1/2}}, y)_{\widetilde{\calM}}  \\
                &\qquad\qquad+ (\dot{\bfX} \cdot \bfnu_{t_{n+1}}, y)_{\widetilde{\calM}}
            \end{aligned}\right]\!
    \end{equation}
    where in the first equality we use \eqref{eq:use_nu}, and in the second we apply 3-stage Gauss--Lobatto quadrature, noting that both $\dot{\bfX} = (\bfX_{t_{n+1}} - \bfX_{t_n})/(t_{n+1} - t_n)$ and $y$ are uniform in time over $T_n$.
    Defining the intermediate normal $\hat{\bfnu}_n := \tfrac{1}{6}(\bfnu_{t_n} + 4\bfnu_{t_{n+1/2}} + \bfnu_{t_{n+1}})$, this can be written as
    \begin{equation}
        \int_{T_n} (\dot{\bfX}\cdot\bfn, y)_\calM
            = ((\bfX_{t_{n+1}} - \bfX_{t_n})\cdot\hat{\bfnu}_n, y)_{\widetilde{\calM}},
    \end{equation}
    aligning with the formulation proposed by Bao \& Zhao~\cite[(3.12)]{Bao_Zhao_2021}.
\end{remark}

\subsubsection{Area dissipation in MCF and SD}

To guarantee area dissipation in both MCF \eqref{eq:mcf_cpg} and SD \eqref{eq:sd_cpg}, the same quadrature rule must be applied to (i)~$\int_{T_n} (\dot{\bfX} \cdot \bfn, y)_\calM$ and $\int_{T_n} (\kappa \bfn, \bfLambda)_\calM$, (ii)~$\int_{T_n} (\nabla_\calM \bfR, \nabla_\calM \bfLambda)_\calM$ and $\int_{T_n} (\nabla_\calM \dot{\bfX}, \nabla_\calM \bfQ)_\calM$, and (iii)~$\int_{T_n} (p\bfn, \bfQ)_\calM$ and $\int_{T_n} (\bfR\cdot\bfn, s)_\calM$.
The final two conditions ensure \Cref{lem:orthogonality} is preserved.

We require also that (iv)~the approximation to $\int_{T_n}(\nabla_\calM\bfX, \nabla_\calM\bfLambda)_\calM$ evaluates to $S(t_{n+1}) - S(t_n)$ when $\bfLambda = \dot{\bfX}$.
As far as we are aware, the integrand $(\nabla_\calM\bfX, \nabla_\calM\dot{\bfX})_\calM$ does not in general reduce to a polynomial in $t$ when evaluated over $\widetilde{\calM}$, rendering this final condition difficult to achieve.
In practice we evaluate this term with a higher-degree quadrature scheme, chosen so that quadrature error is on the order of the nonlinear solver tolerances.





\section{Numerical results}\label{sec:numerics}

In this section, we present numerical experiments that validate the structure-preserving properties of our proposed schemes (\ref{eq:mcf_cpg},~\ref{eq:sd_cpg}) and demonstrate their high-order convergence rates.
Each simulation considers surface flows ($d=2$), discretizing in space with degree-$k$ Lagrange finite elements (CG($k$)) for various $k\in\{1,2,3\}$.
All computations are carried out in \texttt{Firedrake}~\cite{rathgeber2016,ham2023c}, with the CPG time stepping implemented via \texttt{Irksome}~\cite{farrell2020b}.
The code for these experiments is open source and publicly available \cite{codes_repository}.

\subsection{Mean curvature flow (MCF)}

We consider first the proposed MCF discretization \eqref{eq:mcf_cpg}.

\subsubsection{Convergence tests}

We verify expected convergence orders for two initial geometries.
The first is the unit sphere $\calM_0 = \bbS^2$.
In the continuous case, the surface $\calM$ remains spherical under MCF for all $t \in [0,\tfrac{1}{4})$, with radius
\begin{equation}\label{eq:mcf-sphere-exact}
    R(t) = \sqrt{1 - 4t}.
\end{equation}
We test the convergence of our scheme through the error of the discrete solution with this reference solution at time $t = 0.05$.
The second is an asymmetric perturbation of the ellipsoid, obtained by mapping
each $(x,y,z) \in \widetilde{\calM} = \bbS^2$ to
\begin{equation}\label{eq:mcf-ellipsoid-IC}
    \bfX(x,y,z) =
        \big(2x + 0.5\,yz,\ 1.5\,y + 0.4\,xz,\ z + 0.35\,xy\big).
\end{equation}
This combines an axis-aligned anisotropic stretch with a small quadratic shear, breaking all reflective symmetries.
No exact solution is known for this geometry;
the error is therefore computed by comparison between adjacent refinement levels.

To measure the discrepancy between two closed surfaces $\calM_1,\calM_2$, we use the area-averaged closest-point distance, the ``mean error'' $\calE_M$ provided by the \texttt{dune-meshdist} library~\cite{dune-meshdist}:
\begin{equation}\label{eq:manifold-distance}
    \calE_M(\calM_1,\calM_2)
        := \frac{1}{S(\calM_1)} \int_{\calM_1}\! \mathrm{dist}\big(\bfx,\calM_2\big).
\end{equation}
Here $\mathrm{dist}(\bfx,\calM_2) := \min_{\bfy\in\calM_2}\|\bfx-\bfy\|$ is the minimum $L^2$ distance of $\bfx (\in \calM_1)$ from $\calM_2$, and $S(\calM_1)$ is the surface area of $\calM_1$\footnote{
    We refer the reader to~\cite[Appendix B]{Garcke_et_al_2025b} for the discussion of this error, and~\cite{jiang2024stable} for the relation of $\calE_M$ with the manifold distance~\cite{Zhao2021,Jiang23}.
    The implementation of this discrepancy can be found in \cite{codes_repository}.
    We note that two discrete surfaces $\calM_1$, $\calM_2$ produced by different discretizations generally carry different triangulations, rendering a direct vertex-by-vertex comparison inappropriate.
}.

Since the spatial and temporal errors combine additively, the total expected error of the CG($k$)--CPG($s$) method is $\calO(h^{k+1}+\tau^{2s})$, where $\tau$ measures the timestep (assumed to be uniform).
To probe each contribution in isolation, we refine the other enough so that it does not pollute the measured error.
For the \emph{spatial} test, we vary the icosahedral refinement $h$;
to place the temporal contribution one order below the smallest expected spatial error, we fix $s=1$ in time with $\tau \sim h^{(k+2)/2}$.
For the \emph{temporal} test, we vary the time step $\tau$;
we take a sufficiently fine mesh to ensure the spatial error sits well below the temporal floor, with fixed $k = 4$.

\Cref{fig:mcf-eoc-sphere,fig:mcf-eoc-ellipsoid} report the resulting convergence plots on the sphere $\bbS^2$ and on the asymmetric perturbed ellipsoid \eqref{eq:mcf-ellipsoid-IC} respectively.
On the sphere (\Cref{fig:mcf-eoc-sphere}), CG($1$) attains the expected slope~$2$, while CG($2$) and CG($3$) both attain slope~$4$.
The super-convergence of CG($2$) by one extra power is an artifact of the symmetric icosahedral discretization of the unit sphere:
it disappears on the perturbed ellipsoid (\Cref{fig:mcf-eoc-ellipsoid}), where CG($k$) recovers the expected slope~$k+1$ for $k\in\{1,2,3\}$ ($2$, $3$, and $4$ respectively).
Temporally, CPG($s$) attains the expected slope~$2s$ for $s\in\{1,2,3\}$ on both geometries, confirming the predicted high-order temporal accuracy of the proposed scheme beyond the symmetric setting.

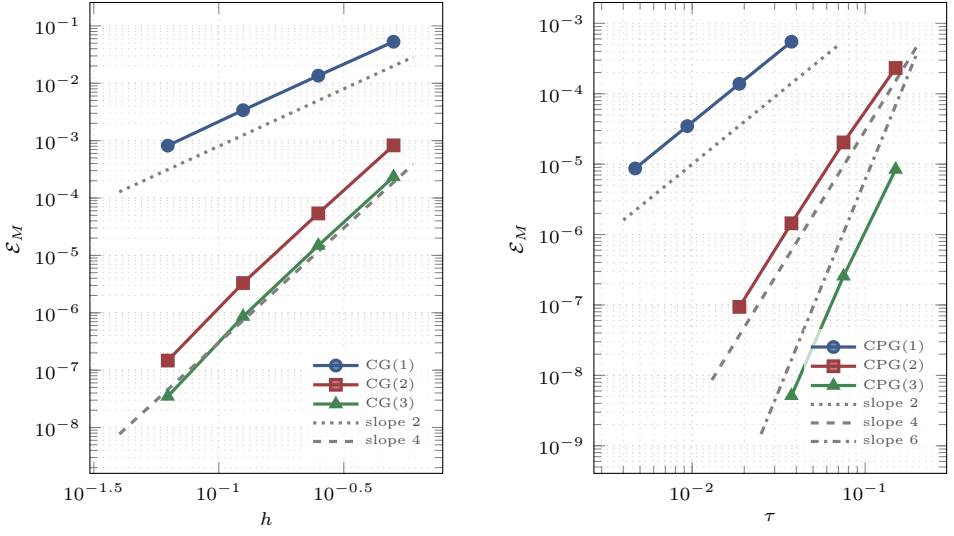
\begin{figure}[!htbp]
    \centering
    \begin{tikzpicture}
        \pgfplotsset{
            every axis/.append style={
                width=0.48\linewidth, height=0.6\linewidth,
                xmode=log, ymode=log,
                xlabel style={font=\scriptsize, yshift=2pt},
                ylabel style={font=\scriptsize, yshift=-2pt},
                tick label style={font=\scriptsize},
                grid=both,
                grid style={dotted, gray!40},
                every axis plot/.append style={very thick},
                legend style={font=\tiny, draw=none, fill opacity=0.7,
                              inner sep=2pt, row sep=-2pt},
                legend cell align=left,
            },
        }
        \begin{groupplot}[
            group style={group size=2 by 1, horizontal sep=2cm},
        ]
        \nextgroupplot[
            xlabel={$h$},
            ylabel={$\calE_M$},
            legend pos=south east,
        ]
            \addplot[seabornblue!80!black, mark=*]
                table[x index=0, y index=1]
                {figures/convergence_surface_mcf_space_sphere_meshdist_CG1.dat};
                \addlegendentry{CG(1)}
            \addplot[seabornred!80!black, mark=square*]
                table[x index=0, y index=1]
                {figures/convergence_surface_mcf_space_sphere_meshdist_CG2.dat};
                \addlegendentry{CG(2)}
            \addplot[seaborngreen!80!black, mark=triangle*]
                table[x index=0, y index=1]
                {figures/convergence_surface_mcf_space_sphere_meshdist_CG3.dat};
                \addlegendentry{CG(3)}
            \addplot[domain=0.04:0.6, dotted, gray] {0.08*x^2};
                \addlegendentry{slope 2}
            \addplot[domain=0.04:0.6, dashed, gray] {0.003*x^4};
                \addlegendentry{slope 4}
        \nextgroupplot[
            xlabel={$\tau$},
            ylabel={$\calE_M$},
            legend pos=south east,
        ]
            \addplot[seabornblue!80!black, mark=*]
                table[x index=0, y index=1]
                {figures/convergence_surface_mcf_time_sphere_meshdist_cPG1.dat};
                \addlegendentry{CPG(1)}
            \addplot[seabornred!80!black, mark=square*]
                table[x index=0, y index=1]
                {figures/convergence_surface_mcf_time_sphere_meshdist_cPG2.dat};
                \addlegendentry{CPG(2)}
            \addplot[seaborngreen!80!black, mark=triangle*]
                table[x index=0, y index=1]
                {figures/convergence_surface_mcf_time_sphere_meshdist_cPG3.dat};
                \addlegendentry{CPG(3)}
            \addplot[domain=0.004:0.07, dotted, gray]     {0.1*x^2};
                \addlegendentry{slope 2}
            \addplot[domain=0.013:0.2,  dashed, gray]     {0.3*x^4};
                \addlegendentry{slope 4}
            \addplot[domain=0.025:0.2,  dashdotted, gray] {6*x^6};
                \addlegendentry{slope 6}
        \end{groupplot}
    \end{tikzpicture}
    \caption{Convergence test for MCF scheme \eqref{eq:mcf_cpg} on the
    unit sphere $\mathbb{S}^{2}$, with the discrete solution at
    $T=0.05$ compared to the exact
    solution~\eqref{eq:mcf-sphere-exact} via the mean error
    $\calE_M$~\eqref{eq:manifold-distance}.
    \emph{Left:} spatial convergence.
    \emph{Right:} temporal convergence.}
    \label{fig:mcf-eoc-sphere}
\end{figure}

\begin{figure}[!htbp]
    \centering
    \begin{tikzpicture}
        \pgfplotsset{
            every axis/.append style={
                width=0.48\linewidth, height=0.6\linewidth,
                xmode=log, ymode=log,
                xlabel style={font=\scriptsize, yshift=2pt},
                ylabel style={font=\scriptsize, yshift=-2pt},
                tick label style={font=\scriptsize},
                grid=both,
                grid style={dotted, gray!40},
                every axis plot/.append style={very thick},
                legend style={font=\tiny, draw=none, fill opacity=0.7,
                              inner sep=2pt, row sep=-2pt},
                legend cell align=left,
            },
        }
        \begin{groupplot}[
            group style={group size=2 by 1, horizontal sep=2cm},
        ]
        \nextgroupplot[
            xlabel={$h$},
            ylabel={$\calE_M$},
            legend pos=south east,
        ]
            \addplot[seabornblue!80!black, mark=*]
                table[x index=0, y index=1]
                {figures/convergence_surface_mcf_space_ellipsoid_meshdist_CG1.dat};
                \addlegendentry{CG(1)}
            \addplot[seabornred!80!black, mark=square*]
                table[x index=0, y index=1]
                {figures/convergence_surface_mcf_space_ellipsoid_meshdist_CG2.dat};
                \addlegendentry{CG(2)}
            \addplot[seaborngreen!80!black, mark=triangle*]
                table[x index=0, y index=1]
                {figures/convergence_surface_mcf_space_ellipsoid_meshdist_CG3.dat};
                \addlegendentry{CG(3)}
            \addplot[domain=0.04:0.6, dotted,     gray] {0.08*x^2};
                \addlegendentry{slope 2}
            \addplot[domain=0.04:0.6, dashed,     gray] {0.005*x^3};
                \addlegendentry{slope 3}
            \addplot[domain=0.04:0.6, dashdotted, gray] {0.003*x^4};
                \addlegendentry{slope 4}
        \nextgroupplot[
            xlabel={$\tau$},
            ylabel={$\calE_M$},
            legend pos=south east,
        ]
            \addplot[seabornblue!80!black, mark=*]
                table[x index=0, y index=1]
                {figures/convergence_surface_mcf_time_ellipsoid_meshdist_cPG1.dat};
                \addlegendentry{CPG(1)}
            \addplot[seabornred!80!black, mark=square*]
                table[x index=0, y index=1]
                {figures/convergence_surface_mcf_time_ellipsoid_meshdist_cPG2.dat};
                \addlegendentry{CPG(2)}
            \addplot[seaborngreen!80!black, mark=triangle*]
                table[x index=0, y index=1]
                {figures/convergence_surface_mcf_time_ellipsoid_meshdist_cPG3.dat};
                \addlegendentry{CPG(3)}
            \addplot[domain=0.0015:0.03, dotted,     gray] {0.15*x^2};
                \addlegendentry{slope 2}
            \addplot[domain=0.003:0.03,  dashed,     gray] {7*x^4};
                \addlegendentry{slope 4}
            \addplot[domain=0.012:0.05,  dashdotted, gray] {200*x^6};
                \addlegendentry{slope 6}
        \end{groupplot}
    \end{tikzpicture}
    \caption{Convergence test for MCF scheme \eqref{eq:mcf_cpg} on the
    asymmetric perturbed ellipsoid \eqref{eq:mcf-ellipsoid-IC}, with
    adjacent-refinement comparison via the mean
    error $\calE_M$~\eqref{eq:manifold-distance}.
    \emph{Left:} spatial convergence.
    \emph{Right:} temporal convergence.}
    \label{fig:mcf-eoc-ellipsoid}
\end{figure}
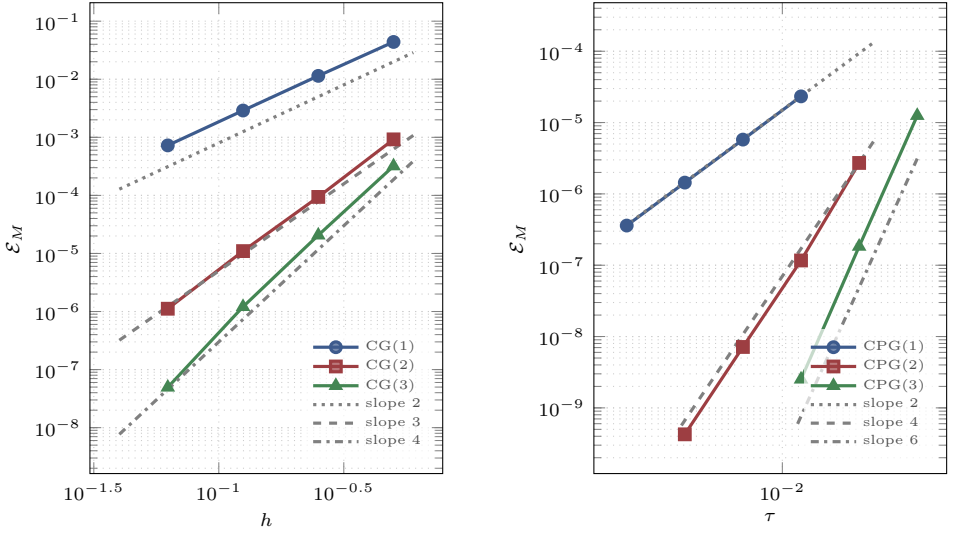

\subsubsection{Dumbbell benchmark \& mesh quality}

We next test our proposed MCF scheme \eqref{eq:mcf_cpg} on a classical benchmark:
the evolution of an axisymmetric ``dumbbell'' surface, parameterized in the same form as in \cite{Gao_Li_Tang_2026,jiang2024stable}:
\begin{equation}
    \bfX(\theta,\varphi) =
    \!\begin{pmatrix}
        \cos\varphi \\
        (0.6\cos^{2}\varphi + 0.4)\,\cos\theta \sin\varphi \\
        (0.6\cos^{2}\varphi + 0.4)\,\sin\theta \sin\varphi
    \end{pmatrix}\!,
    \quad \theta\in[0,2\pi),\quad \varphi\in[0,\pi].
    \label{eq:mcf-dumbbell-IC}
\end{equation}
Under MCF this surface is known first to reduce its surface area, then shrink to a point (the ``blowup'' time at which $S$ approaches zero).
Discretizations that do not include a suitable tangential motion typically suffer severe mesh distortion in this regime~\cite{Gao_Li_Tang_2026}.

We discretize with continuous Galerkin (CG) elements of order $k$ in space and $k$-stage CPG in time, for $k \in \{1,2,3\}$.
The reference triangulation on $\widetilde{\calM}$ is shared across the three runs---a uniform refinement of the unit icosahedral sphere, with $642$ vertices and $1280$ triangles---with the time step fixed at $10^{-5}$.
Each run is advanced until the Newton iteration in the CPG time discretization step fails, coinciding with the blowup.
The blowup times produced by the three schemes are $t_{\text{blowup}} = 9.200 \times 10^{-2}$, $9.094 \times 10^{-2}$ and $9.094 \times 10^{-2}$ respectively.

\Cref{fig:mcf-dumbbell-evolution} shows snapshots of the discrete surface at the initial time $t = 0$, two intermediate times $t = 0.04$ and $t = 0.08$, and at the iteration before the blowup time $t = t_{\text{blowup}}$.
The dumbbell first relaxes its ``waist'' as MCF straightens regions of high mean curvature, then collapses isotropically toward a point.
By $t = 0.08$ the diameter has shrunk, and at $t = t_{\text{blowup}}$ the discrete surface is essentially a spherical point, the expected end state of MCF before the singularity.

\Cref{fig:mcf-dumbbell-comparison} reports two quantitative diagnostics:
the normalized surface area $S(t)/S(0)$, and a mesh quality indicator $r_h(t)$.
This latter quantity tracks the worst-case area distortion of an element relative to its initial state, defined at a time $t$ element-wise as
\begin{equation}\label{eq:rh-def}
    r_h(t) := \frac{\max_{K \in \calK^h} \sup_{\bfx \in K} \calJ_t(\bfx)}{\min_{K \in \calK^h} \inf_{\bfx \in K} \calJ_t(\bfx)},
    \qquad
    \calJ_t(\bfx) := \sqrt{\frac{J_t(\bfx)}{J_0(\bfx)}},
\end{equation}
where $J_t(\bfx)$ is the Jacobian of the map $\bfX_t : \widetilde{\calM} \to \calM_t$, defined within each cell $K \subset \widetilde{\calM}$ \eqref{eq:jacobian};
correspondingly $J_0(\bfx)$ is the Jacobian of the initial mesh\footnote{
    In practice, the suprema $\sup_{\bfx \in K}$ and infima $\inf_{\bfx \in K}$ cannot be evaluated exactly.
    They are evaluated in our code as the suprema and infima over quadrature points.
}.
All three schemes drive the surface area down toward zero before the blowup time, and the mesh-quality ratio remains modest ($r_h \le 2.5$) throughout, confirming that the MDR tangential motion keeps the triangulation well-shaped until the singular collapse.

\begin{figure}[!htbp]
    \centering
    \setlength{\tabcolsep}{0pt}
    \renewcommand{\arraystretch}{0}
    \newcommand{\dbsnapgap}{\\[12pt]}
    \newcommand{\dbbottomgap}{\\[18pt]}
    \newcommand{\dblabel}[1]{\raisebox{0.42\height}{\rotatebox{90}{\scriptsize #1}}}
    \begin{tabular}{@{}c@{\hspace{15pt}}c@{\hspace{15pt}}c@{\hspace{15pt}}c@{}}
        \dblabel{\;\;$t=0$} & &
        \includegraphics[width=0.24\linewidth]{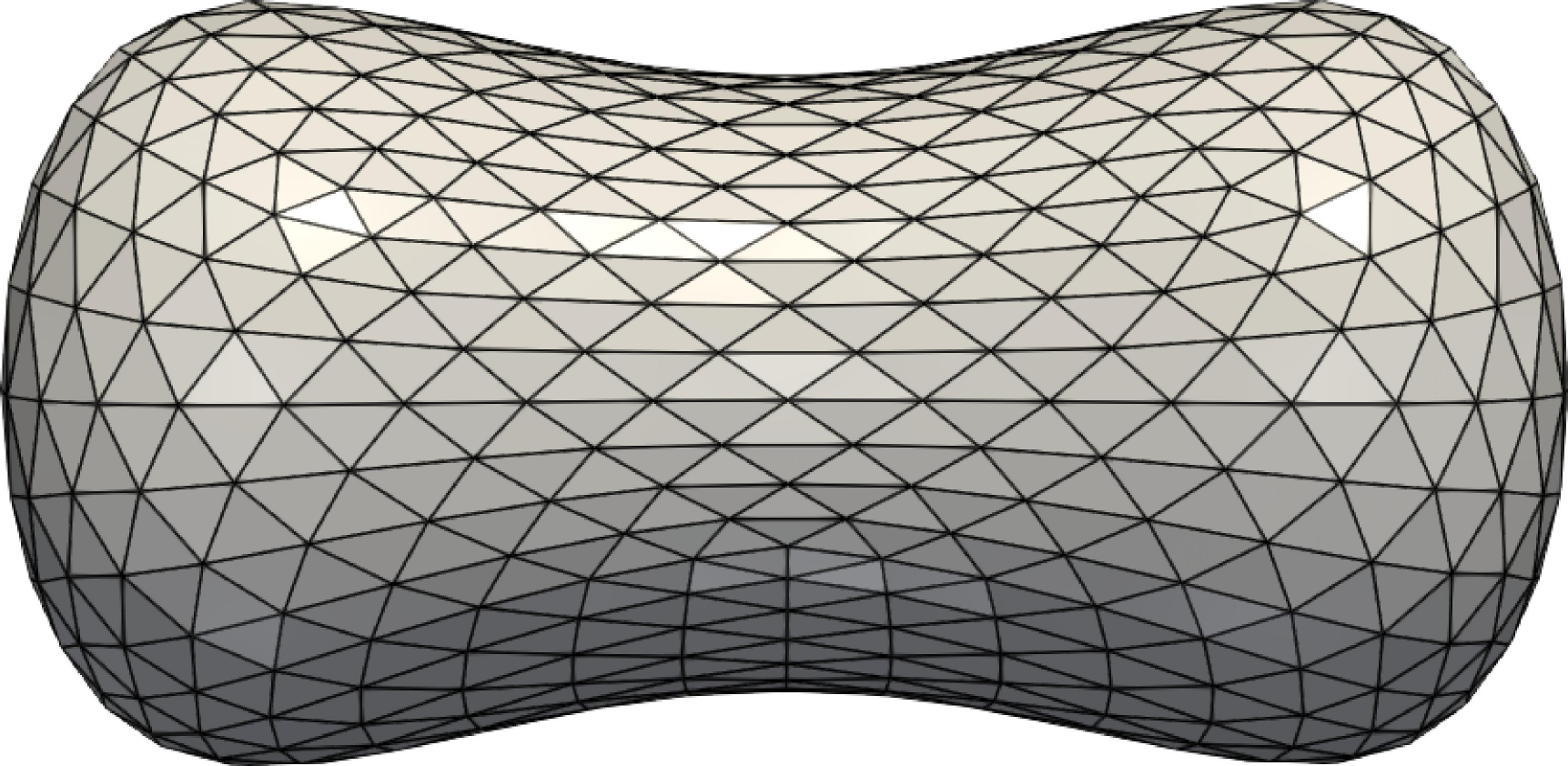} & \dbsnapgap
        & \scriptsize CG(1)--CPG(1) & \scriptsize CG(2)--CPG(2) & \scriptsize CG(3)--CPG(3) \\[2pt]
        \dblabel{\;\;$t=0.04$} &
        \includegraphics[width=0.24\linewidth]{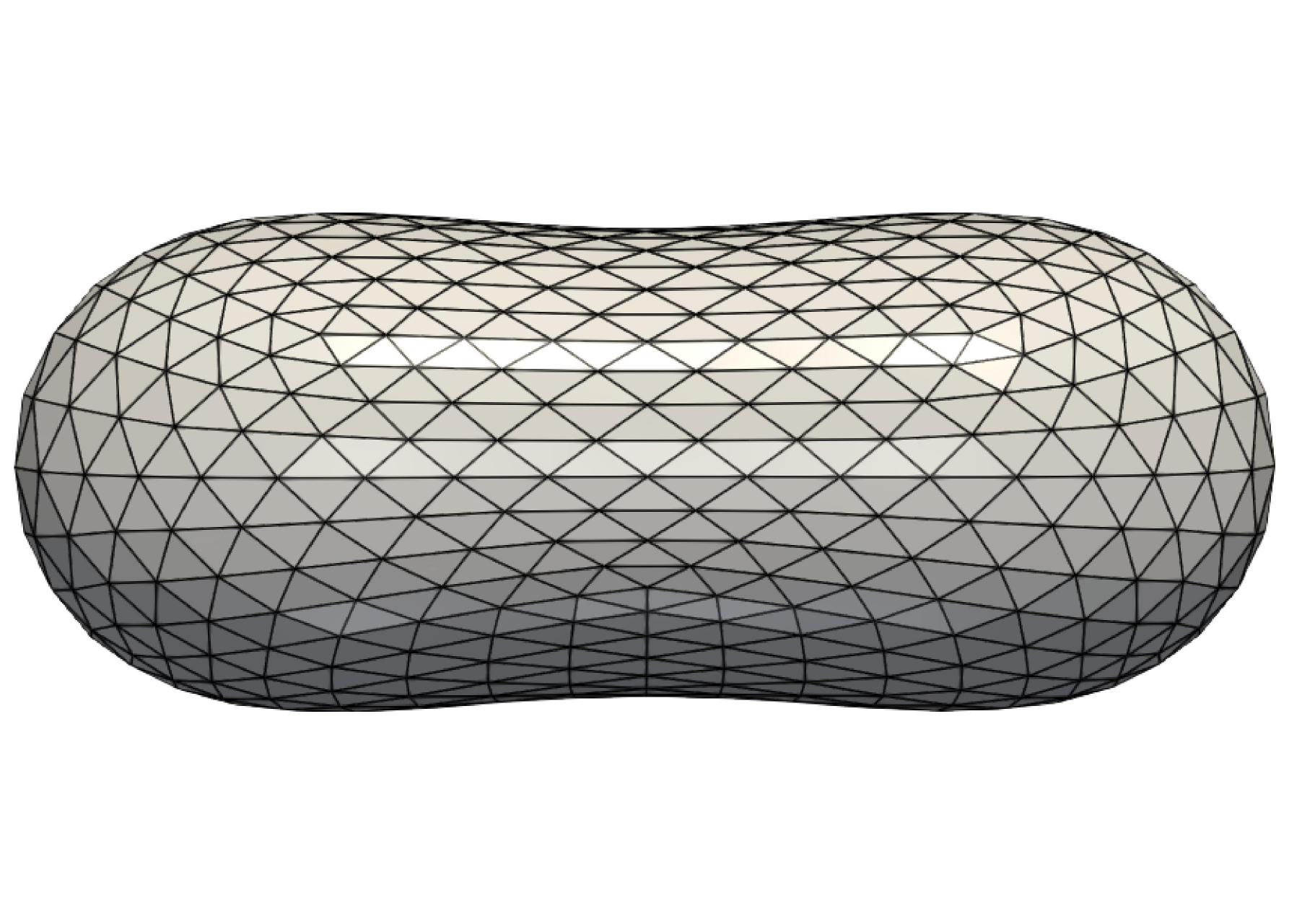} &
        \includegraphics[width=0.24\linewidth]{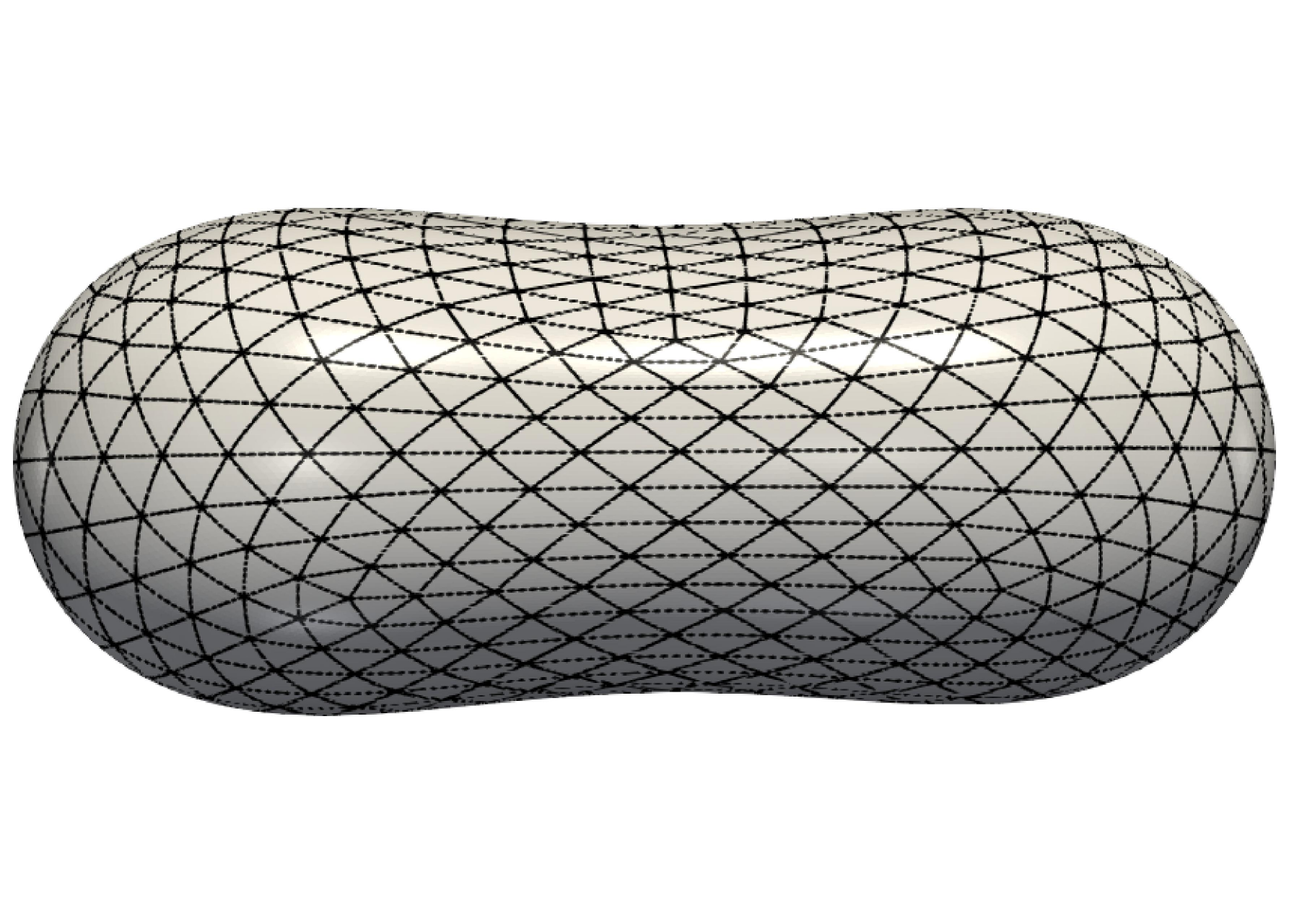} &
        \includegraphics[width=0.24\linewidth]{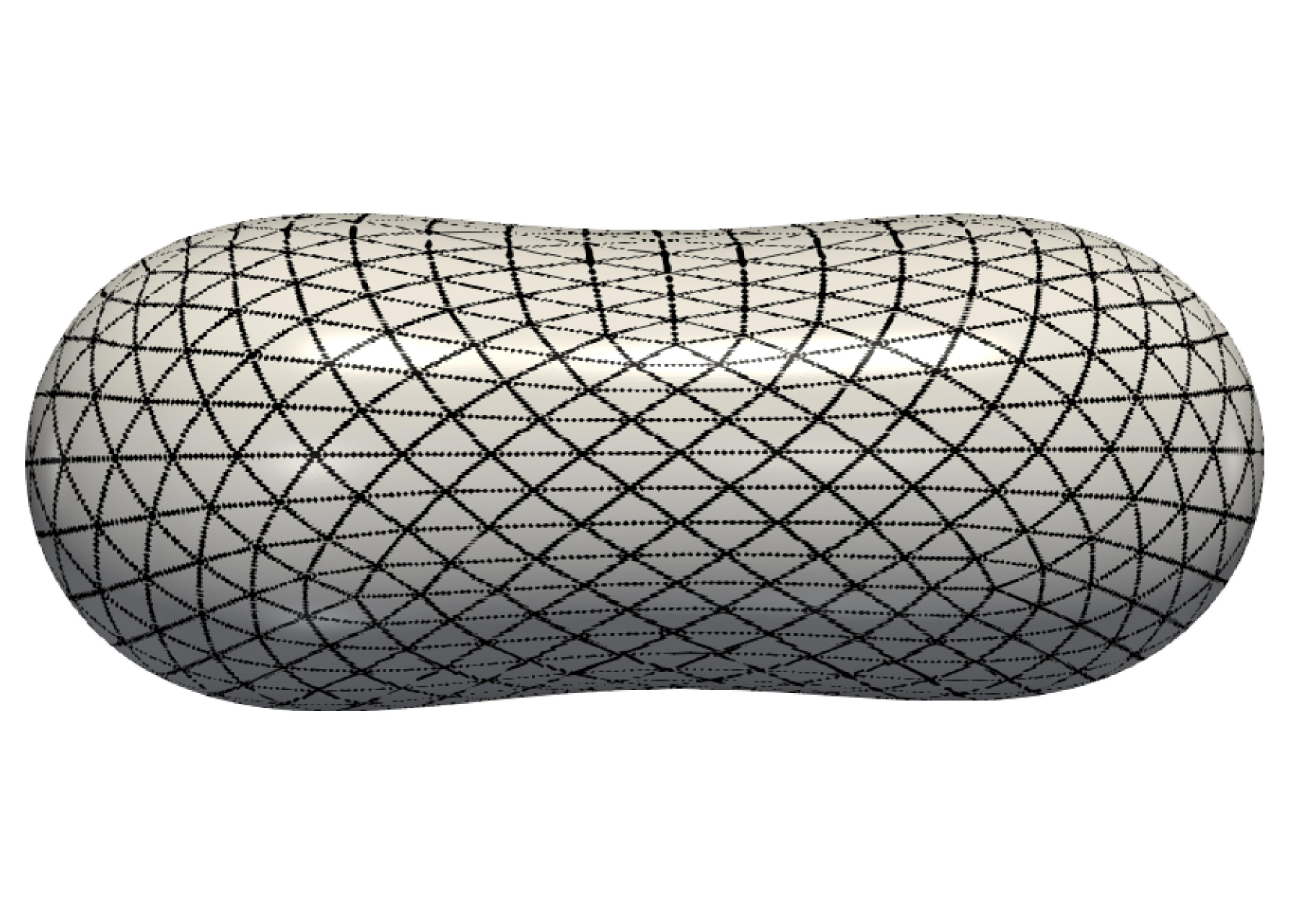} \dbsnapgap
        \dblabel{\;\;$t=0.08$} &
        \includegraphics[width=0.24\linewidth]{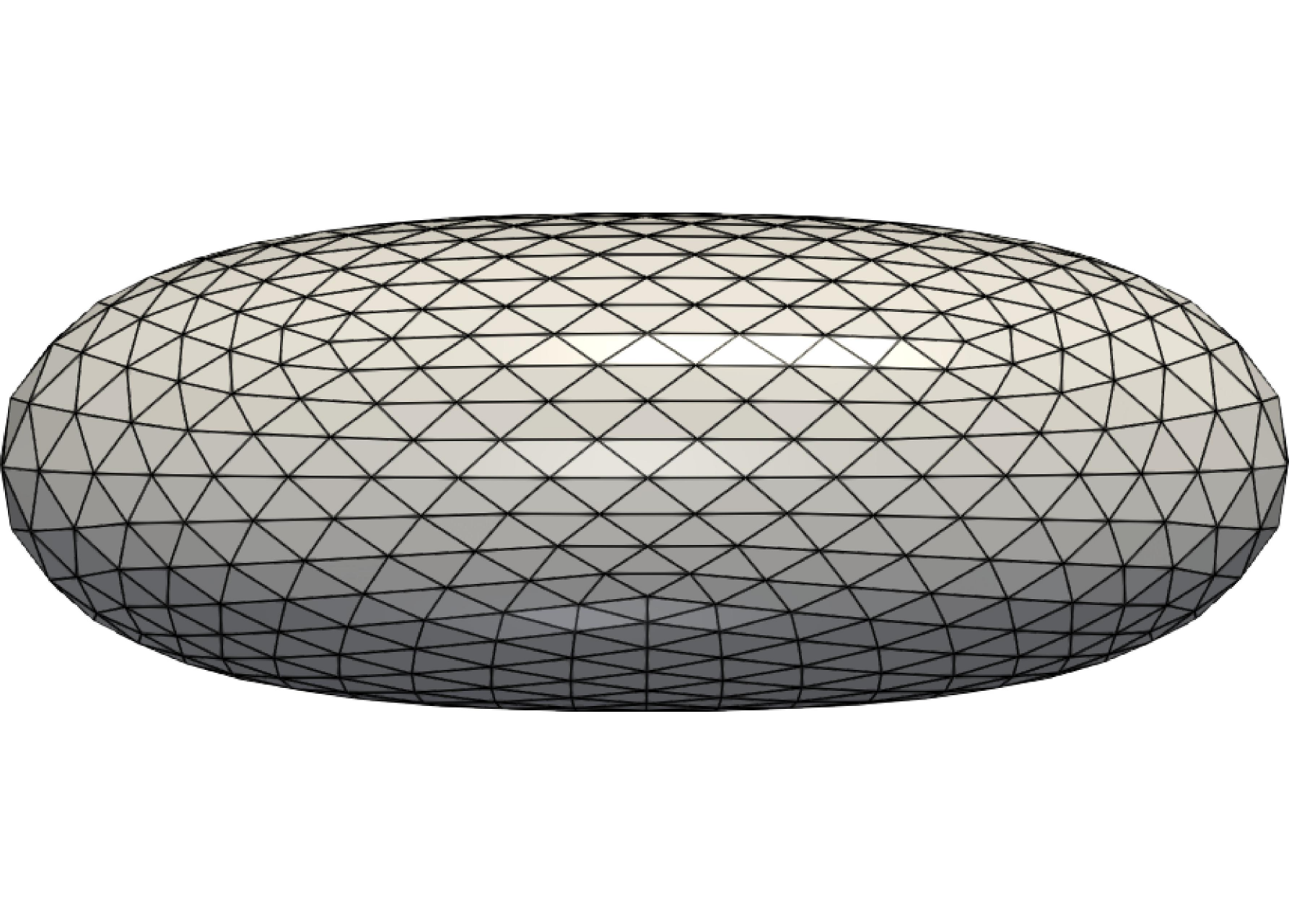} &
        \includegraphics[width=0.24\linewidth]{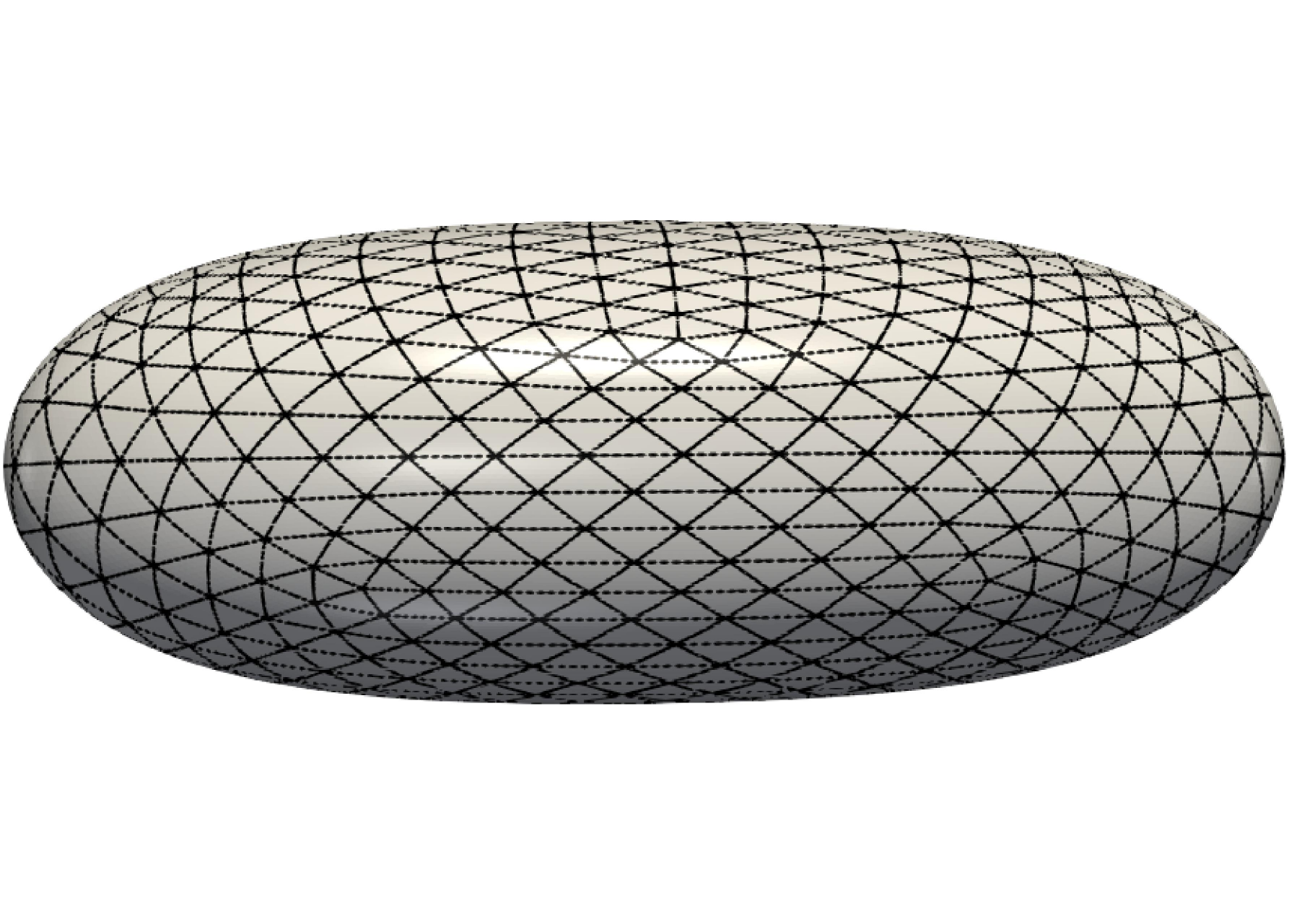} &
        \includegraphics[width=0.24\linewidth]{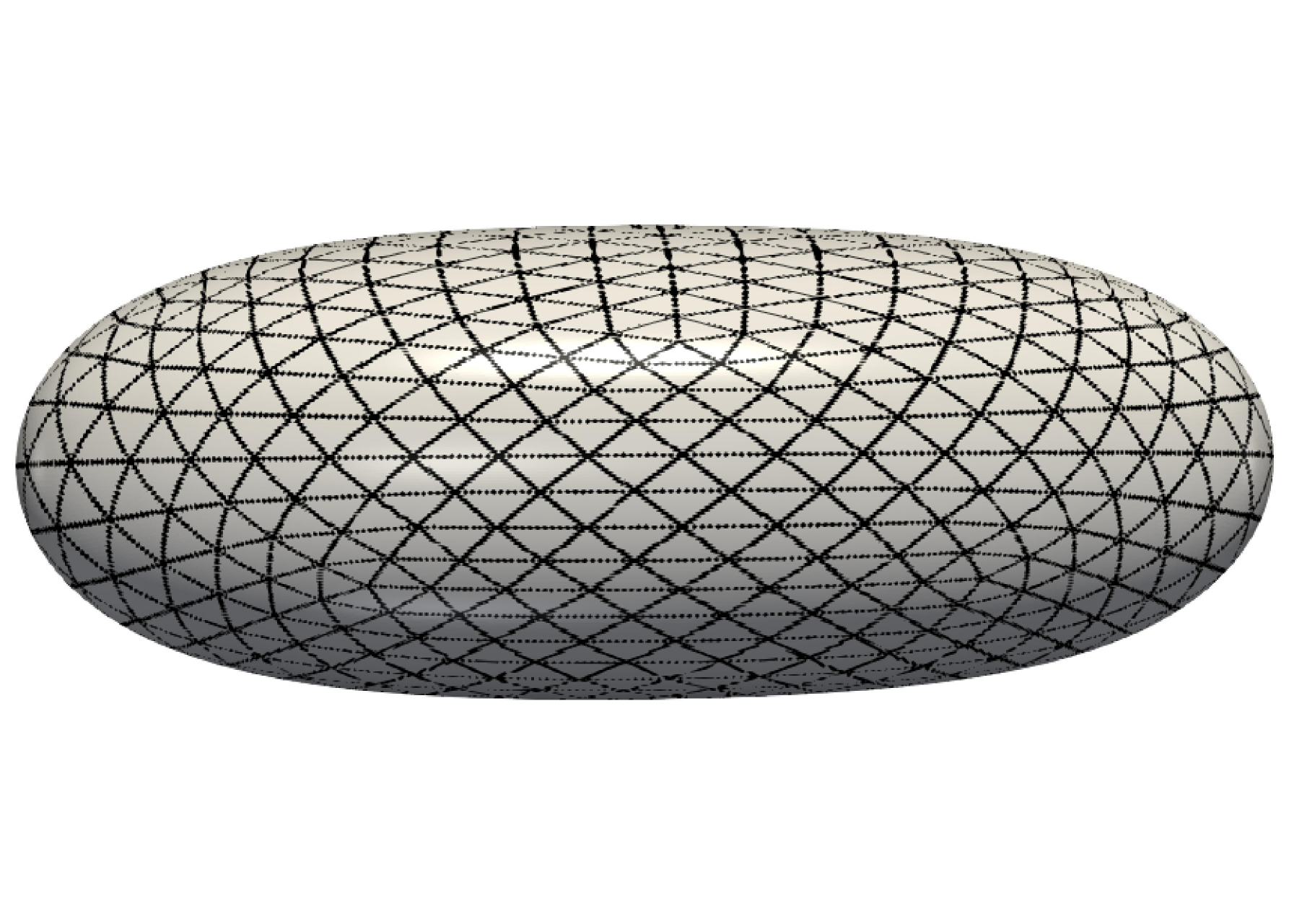} \dbsnapgap
        \dblabel{$t\!=\!t_{\text{blowup}}$} &
        \includegraphics[width=0.24\linewidth]{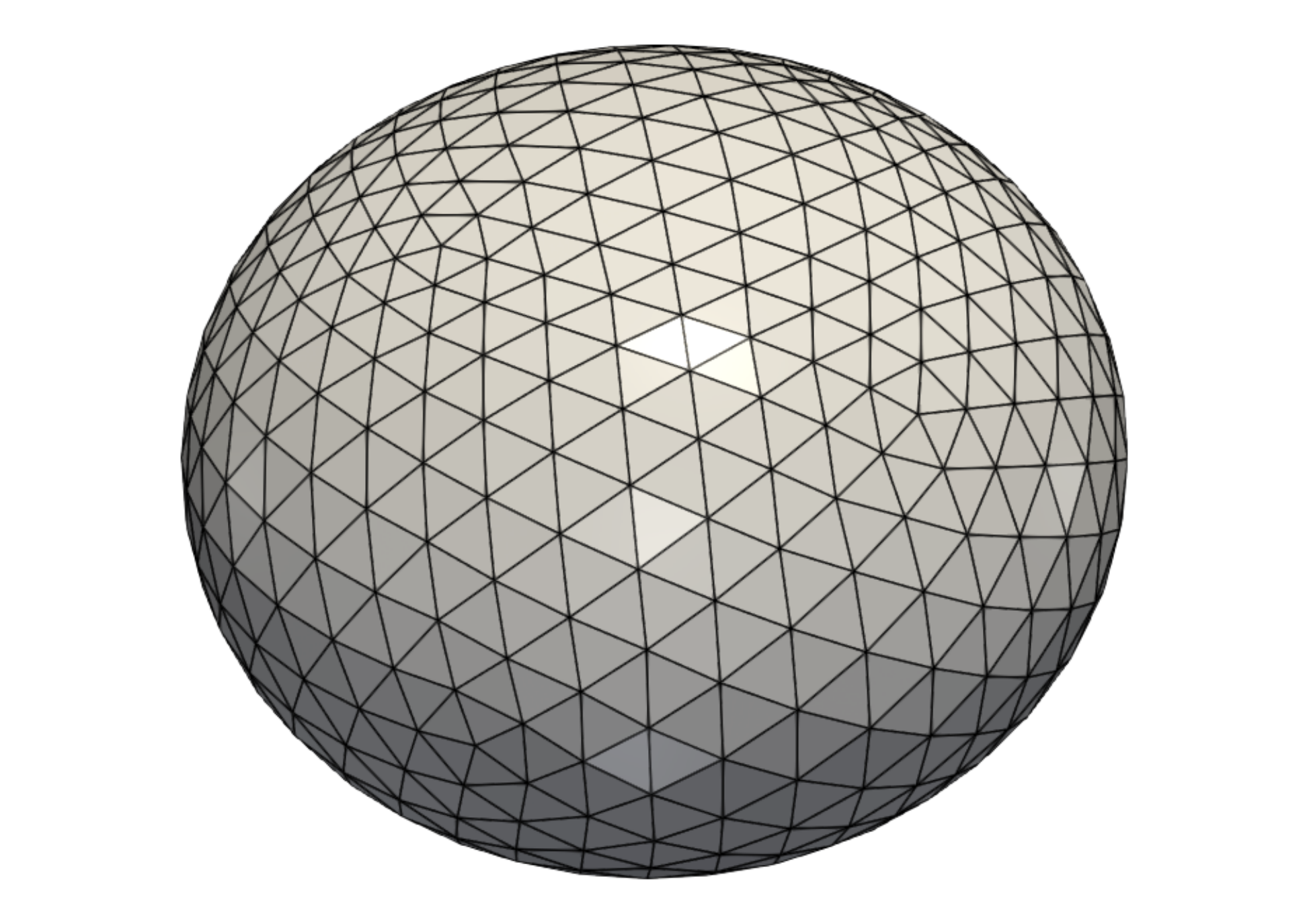} &
        \includegraphics[width=0.24\linewidth]{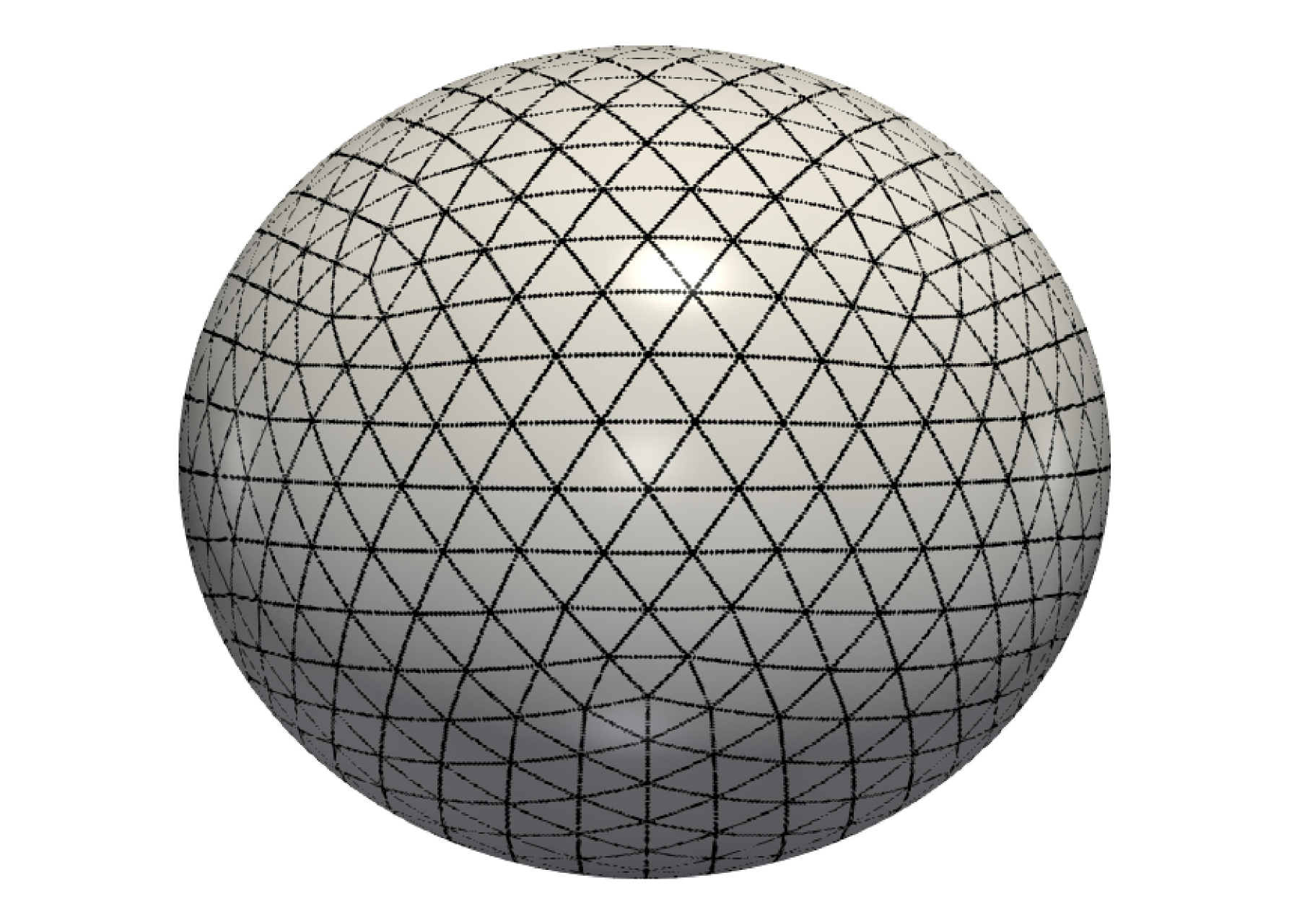} &
        \includegraphics[width=0.24\linewidth]{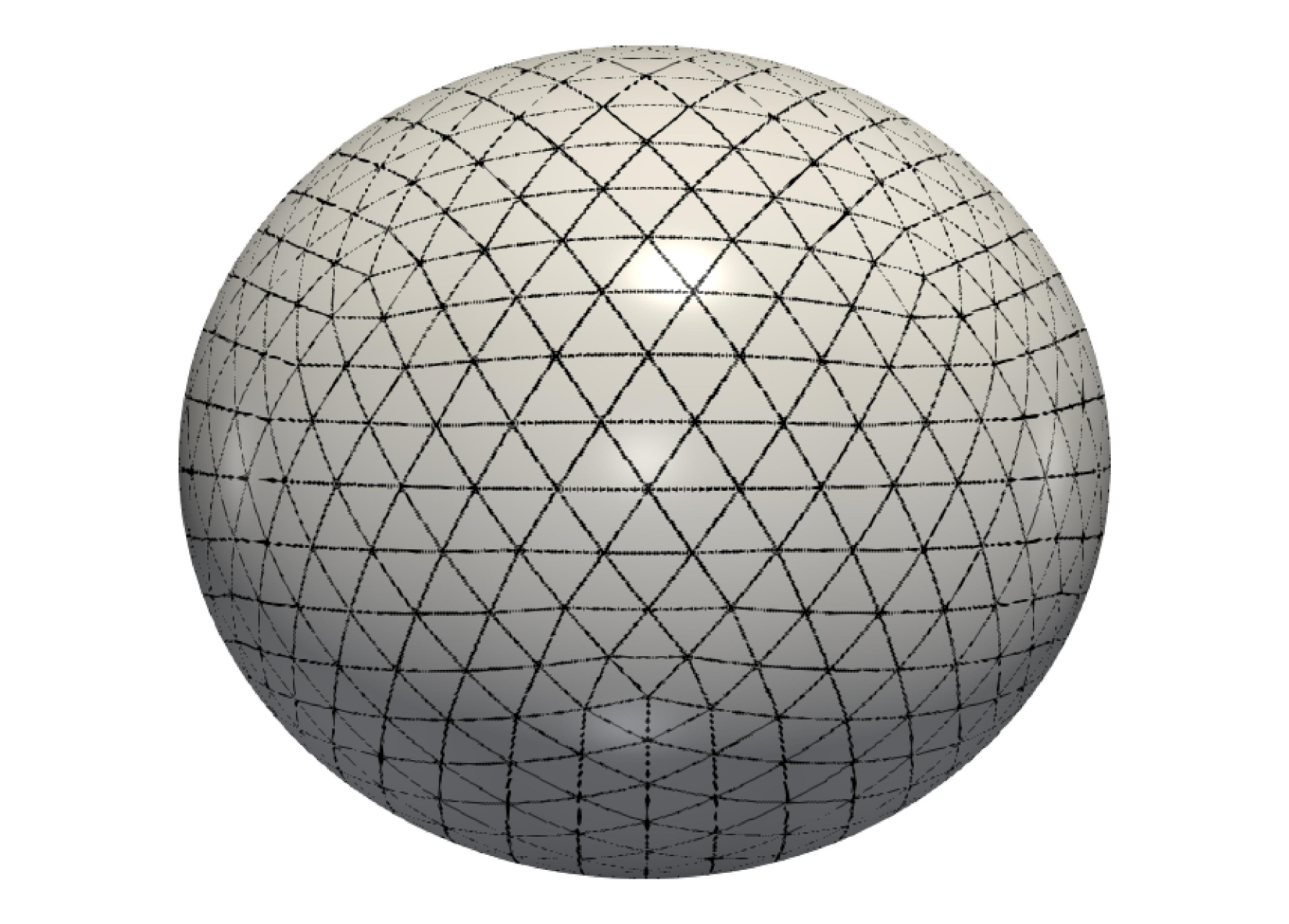} \dbbottomgap
        & \scriptsize $t_{\text{blowup}}\!=\!0.09200$ & \scriptsize $t_{\text{blowup}}\!=\!0.09094$ & \scriptsize $t_{\text{blowup}}\!=\!0.09094$ \\
    \end{tabular}
    \caption{MCF of the dumbbell \eqref{eq:mcf-dumbbell-IC} as computed by our proposed scheme \eqref{eq:mcf_cpg}. The last row of images are not to scale with the others.}
    \label{fig:mcf-dumbbell-evolution}
\end{figure}

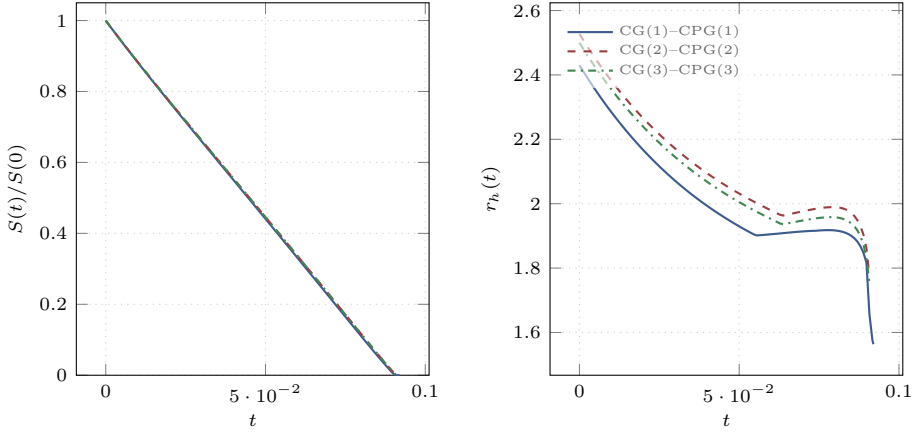
\begin{figure}[!htbp]
    \centering
    \begin{tikzpicture}
        \pgfplotsset{
            every axis/.append style={
                width=0.48\linewidth, height=0.5\linewidth,
                xlabel={$t$},
                xlabel style={font=\scriptsize, yshift=2pt},
                ylabel style={font=\scriptsize, yshift=-2pt},
                tick label style={font=\scriptsize},
                grid=both,
                grid style={dotted, gray!40},
                every axis plot/.append style={very thick},
                cycle list={
                    {seabornblue!80!black, thick, solid},
                    {seabornred!80!black, thick, dashed},
                    {seaborngreen!80!black, thick, dashdotted}
                },
            },
        }
        \begin{groupplot}[
            group style={group size=2 by 1, horizontal sep=1.6cm},
        ]
        \nextgroupplot[
            ylabel={$S(t)/S(0)$},
            ymin=0, ymax=1.05,
        ]
            \addplot table[x=t, y=S_norm]{figures/dumbbell_cg1_hist.dat};
            \addplot table[x=t, y=S_norm]{figures/dumbbell_cg2_hist.dat};
            \addplot table[x=t, y=S_norm]{figures/dumbbell_cg3_hist.dat};
        \nextgroupplot[
            ylabel={$r_{h}(t)$},
            legend pos=north west,
            legend cell align=left,
          legend style={font=\tiny, draw=none, fill opacity=0.6,
                          inner sep=2pt, row sep=-2pt},
        ]
            \addplot table[x=t, y=rh]{figures/dumbbell_cg1_hist.dat};
                \addlegendentry{CG(1)--CPG(1)}
            \addplot table[x=t, y=rh]{figures/dumbbell_cg2_hist.dat};
                \addlegendentry{CG(2)--CPG(2)}
            \addplot table[x=t, y=rh]{figures/dumbbell_cg3_hist.dat};
                \addlegendentry{CG(3)--CPG(3)}
        \end{groupplot}
    \end{tikzpicture}
    \caption{Quantitative diagnostics for MCF computed by
    the scheme \eqref{eq:mcf_cpg}. \emph{Left:} normalized surface
    area $S(t)/S(0)$. \emph{Right:} mesh-quality indicator $r_h(t)$.}
    \label{fig:mcf-dumbbell-comparison}
\end{figure}

\subsection{Surface diffusion (SD)}\label{sec:sd-cuboid}

We now consider the proposed SD discretization \eqref{eq:sd_cpg}.

\subsubsection{Convergence tests}

To the best of our knowledge, no exact solution of SD is available on a general surface.
Taking the initial geometry $\calM_0$ to be the same perturbed ellipsoid initial geometry as considered in the MCF case above \eqref{eq:mcf-ellipsoid-IC}, we test both the spatial and temporal errors with the same strategy:
computation by comparison on adjacent refinement levels, again under the mean error $\calE_M$ \eqref{eq:manifold-distance}.

\Cref{fig:sd-eoc} shows the resulting convergence plots.
Spatially, CG($k$) attains the expected slope~$k+1$ for $k\in\{1,2,3\}$ ($2$, $3$ and $4$ respectively).
Temporally, CPG($s$) attains the expected slope~$2s$ for $s\in\{1,2\}$ ($2$ and $4$ respectively);
at $s=3$, we achieve only a slope~$5$ before solver tolerances and roundoff errors dominate.

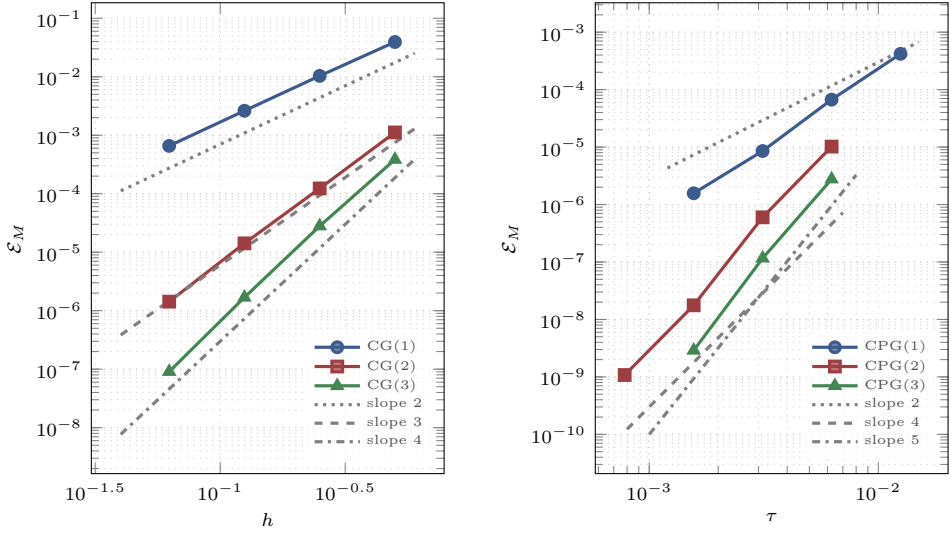
\begin{figure}[!htbp]
    \centering
    \begin{tikzpicture}
        \pgfplotsset{
            every axis/.append style={
                width=0.48\linewidth, height=0.6\linewidth,
                xmode=log, ymode=log,
                xlabel style={font=\scriptsize, yshift=2pt},
                ylabel style={font=\scriptsize, yshift=-2pt},
                tick label style={font=\scriptsize},
                grid=both,
                grid style={dotted, gray!40},
                every axis plot/.append style={very thick},
                legend style={font=\tiny, draw=none, fill opacity=0.7,
                              inner sep=2pt, row sep=-2pt},
                legend cell align=left,
            },
        }
        \begin{groupplot}[
            group style={group size=2 by 1, horizontal sep=2cm},
        ]
        \nextgroupplot[
            xlabel={$h$},
            ylabel={$\calE_M$},
            legend pos=south east,
        ]
            \addplot[seabornblue!80!black, mark=*]
                table[x index=0, y index=1]
                {figures/convergence_surface_sd_space_meshdist_CG1.dat};
                \addlegendentry{CG(1)}
            \addplot[seabornred!80!black, mark=square*]
                table[x index=0, y index=1]
                {figures/convergence_surface_sd_space_meshdist_CG2.dat};
                \addlegendentry{CG(2)}
            \addplot[seaborngreen!80!black, mark=triangle*]
                table[x index=0, y index=1]
                {figures/convergence_surface_sd_space_meshdist_CG3.dat};
                \addlegendentry{CG(3)}
            \addplot[domain=0.04:0.6, dotted,     gray] {0.07*x^2};
                \addlegendentry{slope 2}
            \addplot[domain=0.04:0.6, dashed,     gray] {0.006*x^3};
                \addlegendentry{slope 3}
            \addplot[domain=0.04:0.6, dashdotted, gray] {0.003*x^4};
                \addlegendentry{slope 4}
        \nextgroupplot[
            xlabel={$\tau$},
            ylabel={$\calE_M$},
            legend pos=south east,
        ]
            \addplot[seabornblue!80!black, mark=*]
                table[x index=0, y index=1]
                {figures/convergence_surface_sd_time_meshdist_cPG1.dat};
                \addlegendentry{CPG(1)}
            \addplot[seabornred!80!black, mark=square*]
                table[x index=0, y index=1]
                {figures/convergence_surface_sd_time_meshdist_cPG2.dat};
                \addlegendentry{CPG(2)}
            \addplot[seaborngreen!80!black, mark=triangle*]
                table[x index=0, y index=1]
                {figures/convergence_surface_sd_time_meshdist_cPG3_first3.dat};
                \addlegendentry{CPG(3)}
            \addplot[domain=0.0012:0.015, dotted,     gray] {3*x^2};
                \addlegendentry{slope 2}
            \addplot[domain=0.0008:0.007, dashed,     gray] {300*x^4};
                \addlegendentry{slope 4}
            \addplot[domain=0.001:0.008,  dashdotted, gray] {1e5*x^5};
                \addlegendentry{slope 5}
        \end{groupplot}
    \end{tikzpicture}
    \caption{Convergence test for SD scheme \eqref{eq:sd_cpg},
    measured by the mean error
    $\calE_M$~\eqref{eq:manifold-distance} via
    adjacent-refinement comparison.
    \emph{Left:} spatial convergence.
    \emph{Right:} temporal convergence.}
    \label{fig:sd-eoc}
\end{figure}

\subsubsection{Cuboid benchmark \& mesh quality}

We then test the SD scheme \eqref{eq:sd_cpg} on a benchmark known to develop a finite-time pinch-off singularity:
an $8\times 1\times 1$ cuboid.
The sharp edges and corners are first rounded, after which the bar develops a neck in the middle that pinches at time $t\approx 0.366$~\cite{Barrett_Garcke_Nurnberg_2008a,Gao_Li_Tang_2026,Bao_Zhao_2021,jiang2024stable}.

We similarly discretize with CG of order $k$ in space and $k$-stage CPG in time for $k \in \{1,2,3\}$.
The mesh is a uniform triangulation with 308 vertices and 918 edges, and the time step is again fixed at $10^{-5}$.
Again, each run is advanced until the Newton iteration in the corresponding nonlinear CPG step no longer converges, the natural endpoint at the singularity.
The pinch-off times $t_{\text{pinch}}$ produced by the three schemes are $0.35070$, $0.36724$, $0.36742$ respectively.

\Cref{fig:sd-cuboid-evolution} shows snapshots of the discrete surface at four representative times:
the initial cuboid at $t = 0$, the post-corner-rounding configuration at $t = 0.1$, the necking phase at $t = 0.3$, and the final pre-pinch configuration at $t = t_{\mathrm{pinch}}$ for each scheme (the largest discrete time at which we are still able to find a solution to the nonlinear CPG system).
The three columns correspond to the three schemes;
the four rows are aligned in time.
We observe that all three schemes agree on the general geometry up to the pinch-off time.
\Cref{fig:sd-cuboid-comparison} shows three quantitative diagnostics tracked along the evolution.
The normalized surface area $S(t)/S(0)$, the relative volume change $|\Delta V(t)|/V(0)$, and the mesh-quality indicator $r_h(t)$ defined as in~\eqref{eq:rh-def}.
The choice of these three diagnostics together gives a complete picture: the first plot certifies area dissipation, the second volume conservation, and the third plot shows that the underlying triangulation does not deteriorate during the evolution due to the MDR tangential motion.

\begin{figure}[!htbp]
    \centering
    \setlength{\tabcolsep}{0pt}
    \renewcommand{\arraystretch}{0}
    \newcommand{\snapgap}{\\[24pt]}
    \newcommand{\bottomgap}{\\[30pt]}
    \begin{tabular}{@{}c@{\hspace{15pt}}c@{\hspace{15pt}}c@{\hspace{15pt}}c@{}}
        \rotatebox{90}{\scriptsize $t=0$} & &
        \includegraphics[width=0.27\linewidth]{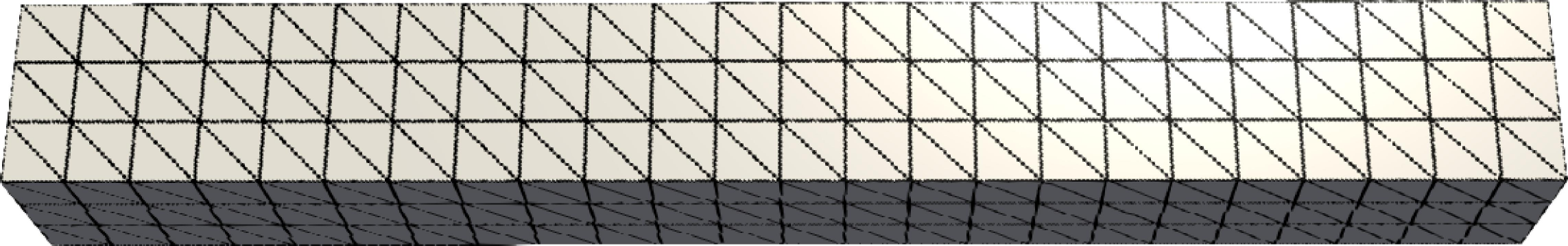} & \snapgap
        & \scriptsize CG(1)--CPG(1) & \scriptsize CG(2)--CPG(2) & \scriptsize CG(3)--CPG(3) \\[2pt]
        \rotatebox{90}{\scriptsize $t=0.1$} &
        \includegraphics[width=0.27\linewidth]{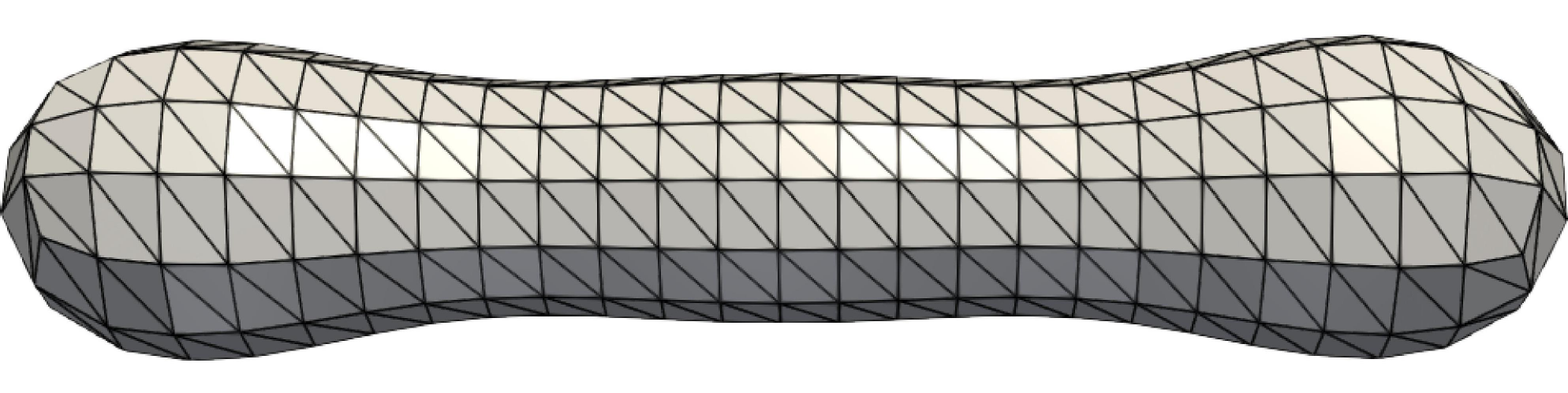} &
        \includegraphics[width=0.27\linewidth]{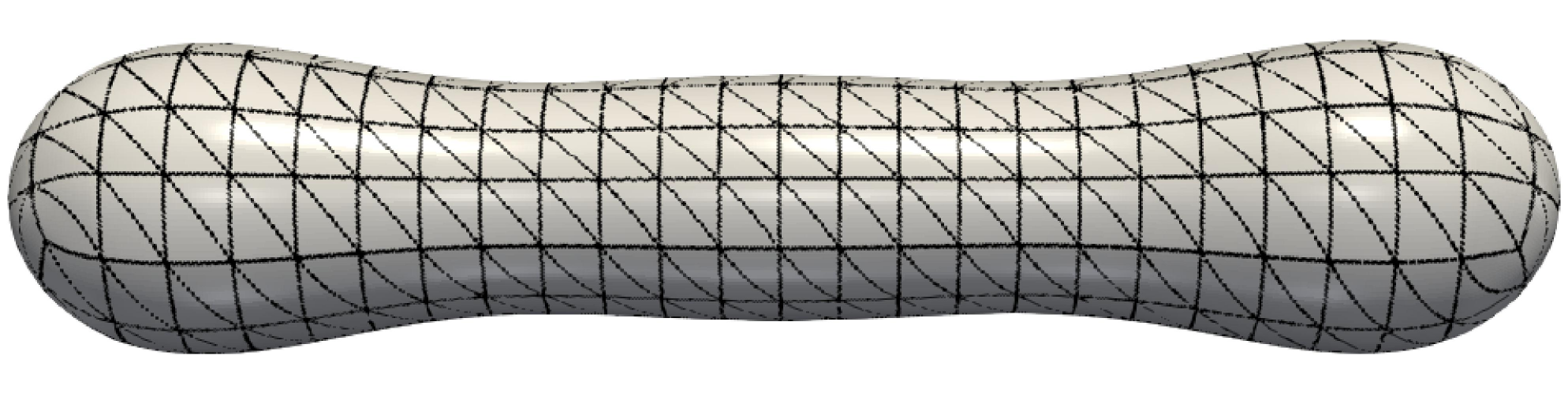} &
        \includegraphics[width=0.27\linewidth]{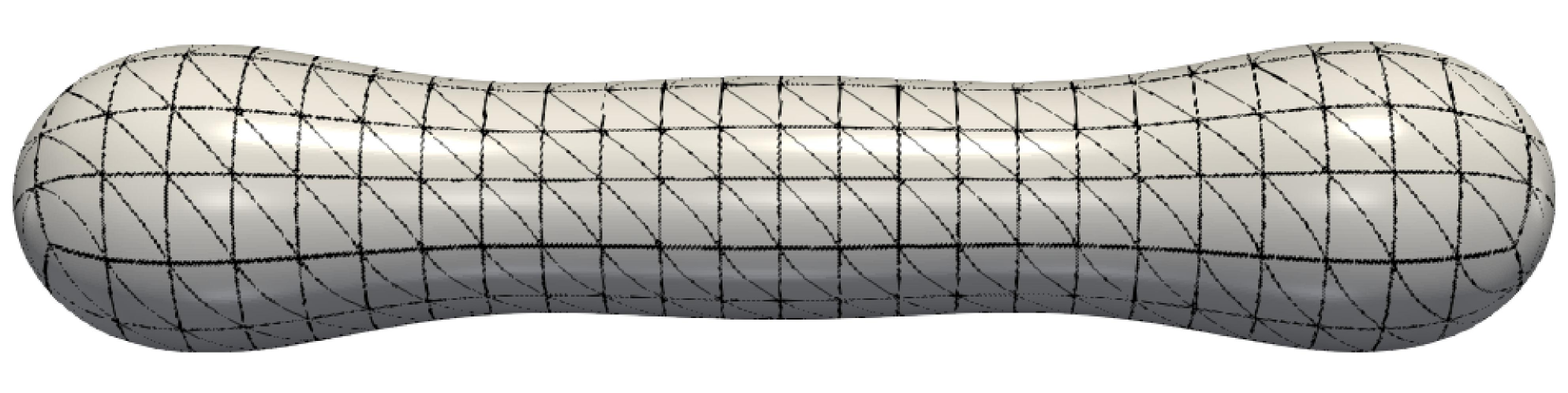} \snapgap
        \rotatebox{90}{\scriptsize $t=0.3$} &
        \includegraphics[width=0.27\linewidth]{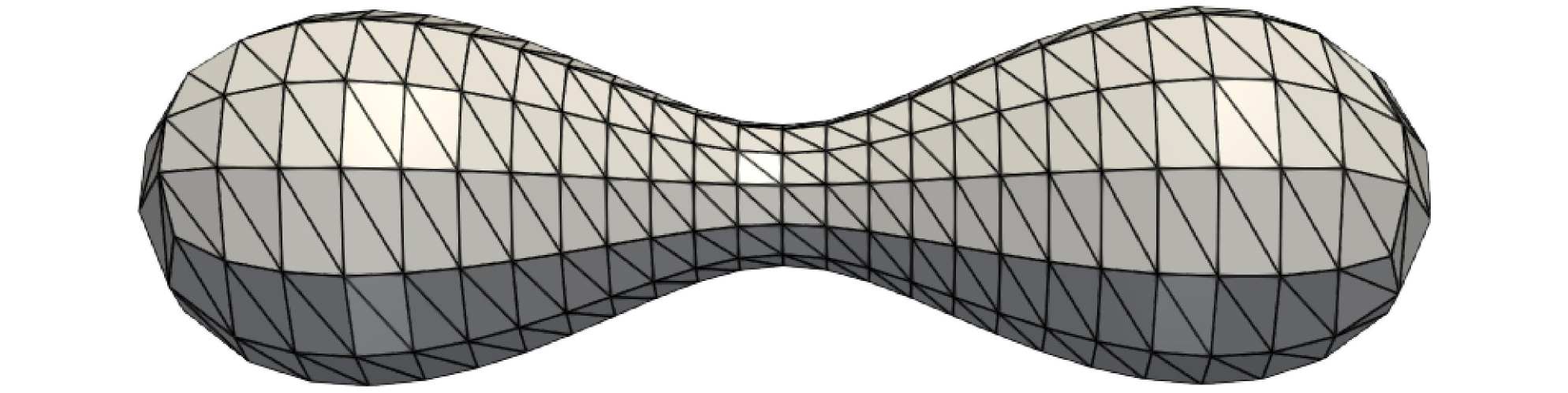} &
        \includegraphics[width=0.27\linewidth]{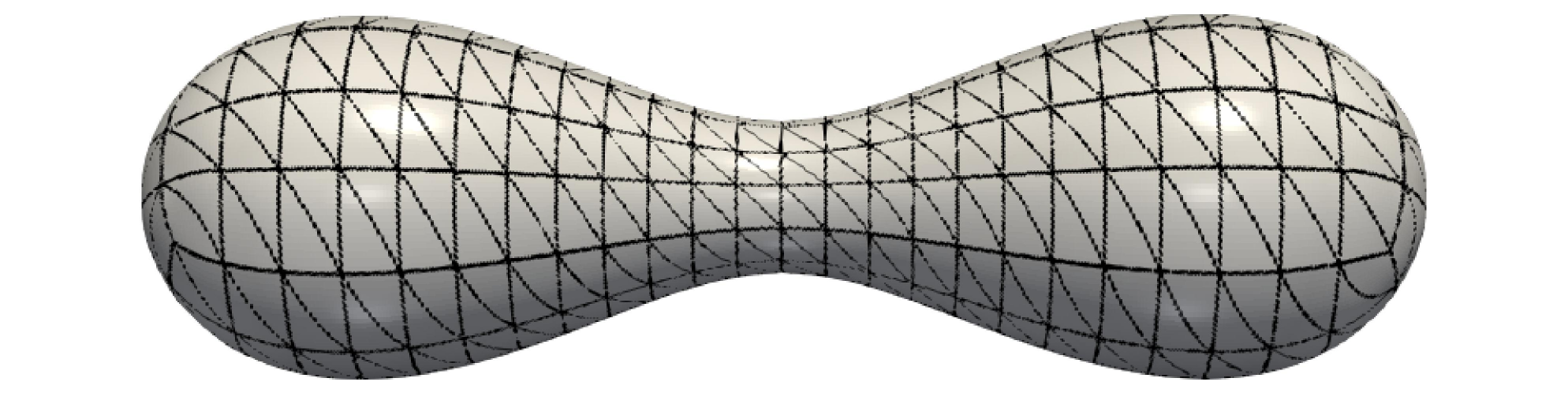} &
        \includegraphics[width=0.27\linewidth]{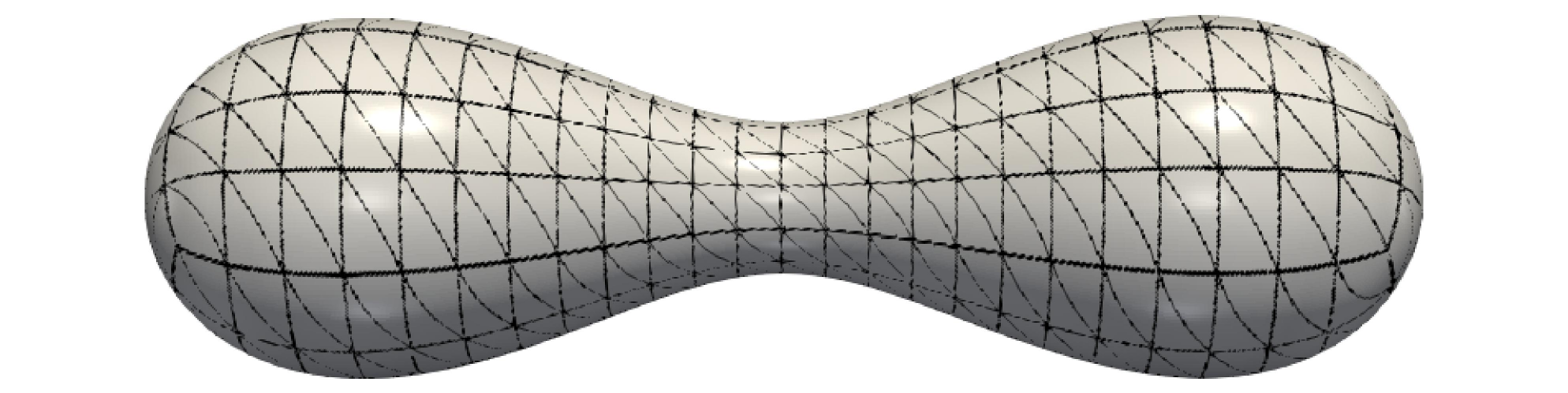} \snapgap
        \rotatebox{90}{\scriptsize \!$t\!=\!t_{\text{pinch}}$} &
        \includegraphics[width=0.27\linewidth]{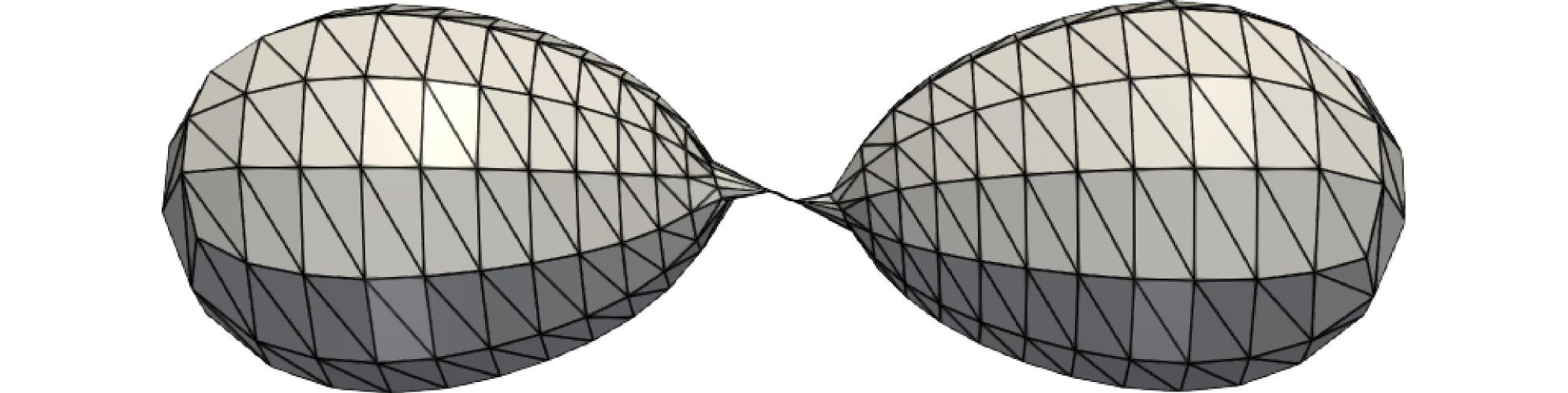} &
        \includegraphics[width=0.27\linewidth]{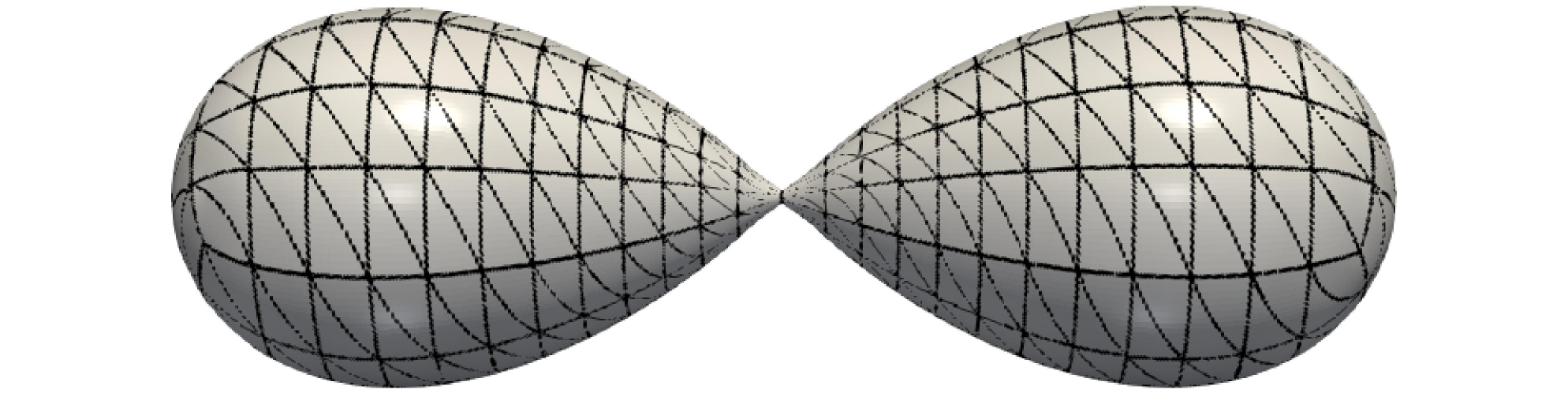} &
        \includegraphics[width=0.27\linewidth]{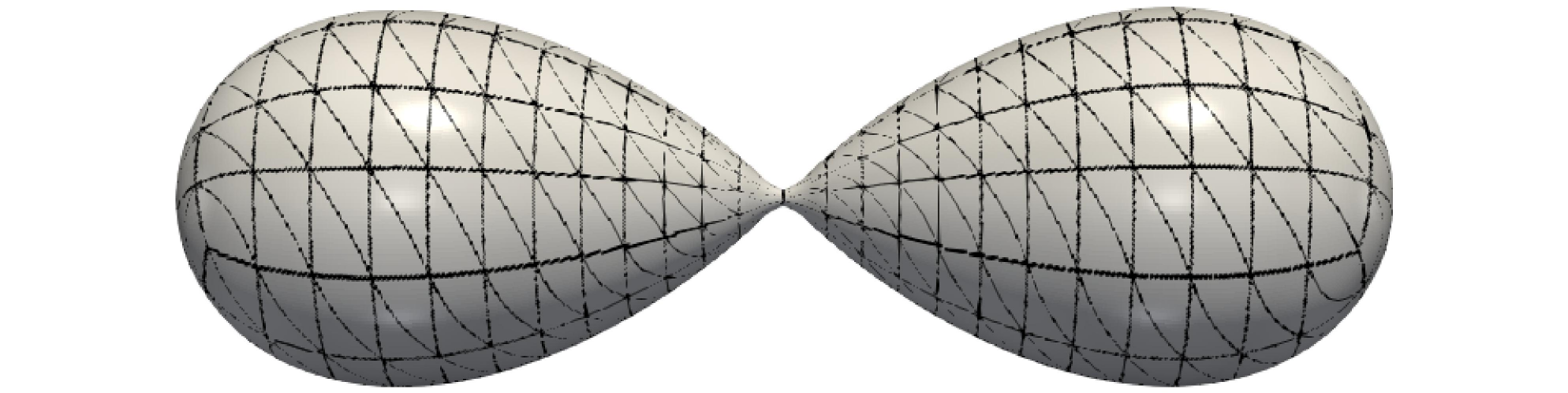} \bottomgap
        & \scriptsize $t_{\text{pinch}}\!=\!0.35070$ & \scriptsize $t_{\text{pinch}}\!=\!0.36724$ & \scriptsize $t_{\text{pinch}}\!=\!0.36742$ \\
    \end{tabular}
    \caption{SD on an $8\times 1\times 1$ cuboid as computed by our proposed scheme \eqref{eq:sd_cpg}}
    \label{fig:sd-cuboid-evolution}
\end{figure}

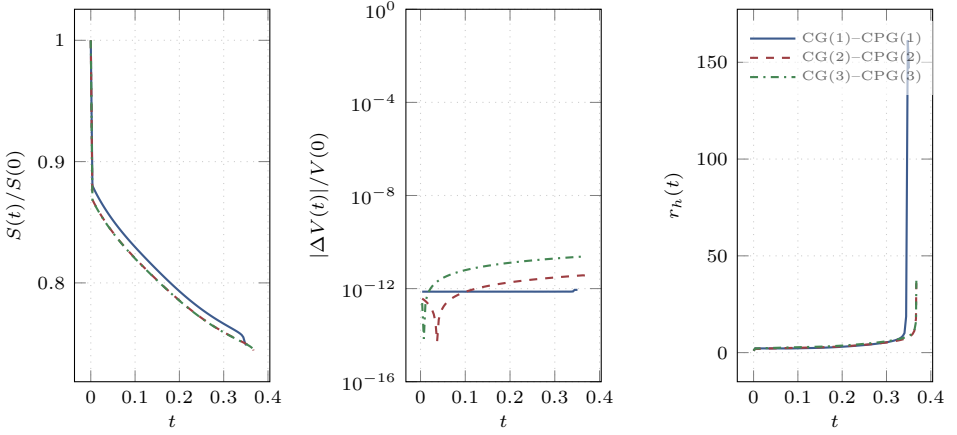
\begin{figure}[!htbp]
    \centering
    \begin{tikzpicture}
        \pgfplotsset{
            every axis/.append style={
                width=0.32\linewidth, height=0.5\linewidth,
                xlabel={$t$},
                xlabel style={font=\scriptsize, yshift=2pt},
                ylabel style={font=\scriptsize, yshift=-2pt},
                tick label style={font=\scriptsize},
                grid=both,
                grid style={dotted, gray!40},
                every axis plot/.append style={very thick},
                cycle list={
                    {seabornblue!80!black, thick, solid},
                    {seabornred!80!black, thick, dashed},
                    {seaborngreen!80!black,thick, dashdotted}
                },
            },
        }
        \begin{groupplot}[
            group style={group size=3 by 1, horizontal sep=1.8cm},
        ]
        \nextgroupplot[ylabel={$S(t)/S(0)$}]
            \addplot table[x=t, y=S_norm]{figures/cuboid_cg1_hist.dat};
            \addplot table[x=t, y=S_norm]{figures/cuboid_cg2_hist.dat};
            \addplot table[x=t, y=S_norm]{figures/cuboid_cg3_hist.dat};
        \nextgroupplot[
            ylabel={$|\Delta V(t)|/V(0)$},
            ymode=log,
            ymin=1e-16, ymax=1,
            ytick={1, 1e-4, 1e-8, 1e-12, 1e-16},
            yticklabels={$10^{0}$, $10^{-4}$, $10^{-8}$, $10^{-12}$, $10^{-16}$},
        ]
            \addplot table[x=t, y expr=abs(\thisrow{dV_raw})]{figures/cuboid_cg1_hist.dat};
            \addplot table[x=t, y expr=abs(\thisrow{dV_raw})]{figures/cuboid_cg2_hist.dat};
            \addplot table[x=t, y expr=abs(\thisrow{dV_raw})]{figures/cuboid_cg3_hist.dat};
        \nextgroupplot[
            ylabel={$r_{h}(t)$},
            legend pos=north west,
            legend cell align=left,
            legend style={font=\tiny, draw=none, fill opacity=0.6,
                          inner sep=2pt, row sep=-2pt},
        ]
            \addplot table[x=t, y=rh]{figures/cuboid_cg1_hist.dat};
                \addlegendentry{CG(1)--CPG(1)}
            \addplot table[x=t, y=rh]{figures/cuboid_cg2_hist.dat};
                \addlegendentry{CG(2)--CPG(2)}
            \addplot table[x=t, y=rh]{figures/cuboid_cg3_hist.dat};
                \addlegendentry{CG(3)--CPG(3)}
        \end{groupplot}
    \end{tikzpicture}
    \caption{Quantitative diagnostics for SD computed by the scheme \eqref{eq:sd_cpg}. \emph{Left:} normalized surface area $S(t)/S(0)$. \emph{Middle:} relative volume change $|\Delta V(t)|/V(0)$ on a logarithmic scale. \emph{Right:} mesh-quality indicator $r_h(t)$.}
    \label{fig:sd-cuboid-comparison}
\end{figure}


\section{Conclusions}\label{sec:conclusion}

We have proposed the first discretizations of mean curvature flow and surface diffusion that are simultaneously arbitrary-order in space, arbitrary-order in time, and structure-preserving.
The construction follows the auxiliary-variable framework of~\cite{Andrews_Farrell_2025b}, starting from an MDR semi-discretization that controls tangential motion to retain mesh quality.
We systematically introduce auxiliary variables $(\kappa, \bfR)$ to replicate the arguments underlying the dissipation of area (and, for SD, conservation of volume) on the discrete level;
at the semi-discrete level, this coincides with the very recently proposed dual-MDR semi-discretization \cite{Gao_Li_Tang_2026}. 
Combining this with a continuous Petrov--Galerkin time discretization yields the schemes \eqref{eq:mcf_cpg} and \eqref{eq:sd_cpg}, which we have shown preserve area dissipation for MCF and SD (mimicking the continuous dissipation rates) and volume conservation for SD, for arbitrary polynomial orders $k$ in space and $s$ in time.

The numerical experiments confirm the expected convergence rates and proven structure-preservation results.
The CG($k$)--CPG($s$) discretization attains convergence rates of $\calO(h^{k+1})$ in space and $\calO(\tau^{2s})$ in time, on both spherical and asymmetric perturbed ellipsoid initial data, for $k, s \in \{1, 2, 3\}$.
On the dumbbell benchmark for MCF and the $8 \times 1 \times 1$ cuboid pinch-off benchmark for SD, the schemes reproduce the expected geometric evolution up to the singularity, certify the structure-preserving properties to within nonlinear solver tolerances and quadrature errors, and retain good mesh quality throughout, inheriting the mesh-stabilizing effect of MDR.

\section*{Acknowledgments}
We would like to thank Rong Tang for informing us about her recent work on the dual-MDR formulation with Guangwei Gao \& Buyang Li~\cite{Gao_Li_Tang_2026}.

\appendix

\section{Evolution of volume and area on parameterized surfaces}
This appendix confirms the identities $\dot{S} = (\nabla_\calM\bfX, \nabla_\calM\dot{\bfX})_\calM$ \eqref{eq:discrete_reynolds} and $\dot{V} = \int_\calM \dot{\bfX}\cdot\bfn$ \eqref{eq:discrete_reynolds_volume}, as used in the proofs of our discrete structures, hold on surfaces parameterized by general continuous finite elements.

Suppose the function space $\bbV$ is defined on a mesh $\calK^h$ of $\widetilde{\calM}$, such that the cells $K \in \calK^h$ partition $\widetilde{\calM}$, while their images $\bfX(K)$ partition $\calM$.
Within each cell $K \in \calK^h$, define the Jacobian
\begin{subequations}\label{eq:jacobian}
\begin{equation}
    J(\bfx) := \sqrt{\det \bfG(\bfx)}
\end{equation}
of the map $\bfX : \widetilde{\calM} \to \calM$, where
\begin{equation}
    \bfG(\bfx) := \nabla_{\widetilde{\calM}}\bfX(\bfx)^\top\nabla_{\widetilde{\calM}}\bfX(\bfx)
\end{equation}
\end{subequations}
is the metric tensor.
We assume no singularities have formed in the flow, such that (i)~$\calM$ is Lipschitz, and (ii)~for each cell $K \subset \widetilde{\calM}$, the restriction $\calM|_{\bfX(K)}$ is smooth.
This latter condition is equivalent to the assumption that $J > 0$.

\subsection{Area identity}\label{app:area_identity_parameterized}

In writing $\dot{S} = (\nabla_\calM\bfX, \nabla_\calM\dot{\bfX})_\calM$ \eqref{eq:discrete_reynolds} we rely on the Reynolds transport theorem~\cite[Theorem~32, Lemma~9]{BGN20}.
It is not immediately clear that this is true for parameterized surfaces that are not globally twice-differentiable (for example, where $\bbV$ is composed of Lagrange finite elements) where the curvature $\Delta_\calM\bfX\cdot\bfn$ is ill-defined between cells.
In such cases, we confirm the identity to be true by evaluating the change of area on each cell $K$ directly, where the map is smooth.

We may evaluate the total surface area $S$ of the target manifold $\calM$ as the sum of the areas $S_K$ of each target cell $\bfX(K)$,
\begin{equation}
    S
        = \sum_{K \in \calK^h} S_K
        = \sum_{K \in \calK^h} \int_{K} J.
\end{equation}
We may evaluate the time derivative of the Jacobian integrand $J$ as
\begin{equation}
    \dot{J}  =  J \tr(\nabla_{\widetilde{\calM}}\bfX\bfG^{-1}\nabla_{\widetilde{\calM}}\dot{\bfX}^\top).
\end{equation}
Thus $\dot{S}$ evaluates as
\begin{multline}
    \dot{S}
        = \sum_{K \in \calK^h} \int_{K} \dot{J}
        = \sum_{K \in \calK^h} \int_{K} J \tr(\nabla_{\widetilde{\calM}}\bfX\bfG^{-1}\nabla_{\widetilde{\calM}}\dot{\bfX}^\top)  \\
        = \sum_{K \in \calK^h} \int_{\bfX(K)} \tr(\nabla_{\calM}\bfX\nabla_{\calM}\dot{\bfX}^\top)
        = \int_\calM \tr(\nabla_\calM\bfX\nabla_\calM\dot{\bfX}^\top),
\end{multline}
where in the third equality we pass from each $K \subset \widetilde{\calM}$ to its image $\bfX(K) \in \calM$.
This is exactly the identity \eqref{eq:discrete_reynolds}.

\subsection{Volume identity}\label{app:volume_identity_parameterized}

Typical statements of the Reynolds transport theorem as used to show the identity $\dot{V} = \int_\calM \dot{\bfX}\cdot\bfn$ \eqref{eq:discrete_reynolds_volume} rely on $C^2$ regularity of $\calM$. 
Similarly to the above (\Cref{app:area_identity_parameterized}) we confirm this holds for $\calM$ parameterized by continuous finite elements.

Denote by $\Omega \subset \bbR^{d+1}$ the region contained within $\calM = \partial\Omega$:
\begin{equation}\label{eq:volume_1}
    V
        = \int_\Omega 1
        = \frac{1}{d+1} \int_\Omega \div\bfX
        = \frac{1}{d+1} \int_\calM \bfX\cdot\bfn
        = \frac{1}{d+1} \sum_{K \in \calK^h} \int_{\bfX(K)} \bfX\cdot\bfn,
\end{equation}
where in the second equality we note $\div\bfX = d+1$ (where with a slight abuse of notation, $\bfX$ denotes the general position vector in $\bbR^{d+1}$), in the fourth we use the Lipschitz regularity of $\calM$ to apply the divergence theorem, and in the final equality we partition $\calM$ into cells $\bfX(K)$.
Restricted to $\bfX(K)$, the surface transport theorem for smooth $\bfF : \bfX(K) \to \bbR^{d+1}$ asserts that
\begin{subequations}
\begin{equation}
    \frac{\rmd}{\rmd t} \int_{\bfX(K)} \bfF\cdot\bfn
        = \int_{\bfX(K)} (\dot{\bfF} + (\div_\calM \bfF)\dot{\bfX}) \cdot \bfn
            + \int_{\partial\bfX(K)} (\bfF \wedge \dot{\bfX}) \cdot \bft.
\end{equation}
Taking $\bfF = \bfX$,
\begin{equation}\label{eq:volume_2}
    \frac{\rmd}{\rmd t} \int_{\bfX(K)} \bfX\cdot\bfn
        = (d+1) \int_{\bfX(K)} \dot{\bfX}\cdot\bfn
            + \int_{\partial\bfX(K)} (\bfX \wedge \dot{\bfX}) \cdot \bft;
\end{equation}
\end{subequations}
noting the tangential divergence $\div_\calM \bfX = d$.
Composing (\ref{eq:volume_1},~\ref{eq:volume_2}):
\begin{equation}
    \dot{V}
        = \sum_{K \in \calK^h} \!\left[\int_{\bfX(K)}\dot{\bfX}\cdot\bfn
            + \frac{1}{d+1}\int_{\partial\bfX(K)} (\bfX \wedge \dot{\bfX}) \cdot \bft\right]\!
        = \int_\calM \dot{\bfX}\cdot\bfn,
\end{equation}
where in the second equality we observe that, on each facet, the terms $\int_{\partial\bfX(K)} (\bfX \wedge \dot{\bfX}) \cdot \bft$ contributed from the cells either side cancel.

\bibliographystyle{siamplain}
\bibliography{references}

@article{Alvarez1993,
  title     = {Axioms and fundamental equations of image processing},
  author    = {Alvarez, Luis and Guichard, Fr{\'e}d{\'e}ric and Lions, Pierre-Louis and Morel, Jean-Michel},
  journal   = {Archive for Rational Mechanics and Analysis},
  volume    = {123},
  number    = {3},
  pages     = {199--257},
  year      = {1993},
  publisher = {Springer},
  doi       = {10.1007/BF00375127},
}

@misc{Andrews_Farrell_2025a,
  title     = {Conservative and dissipative discretisations of multi-conservative {ODEs} and {GENERIC} systems},
  doi       = {10.48550/arXiv.2511.23266},
  publisher = {arXiv},
  author    = {Andrews, B. D. and Farrell, P. E.},
  month     = nov,
  year      = {2025},
  note      = {arXiv manuscript}
}

@article{Andrews_Farrell_2025b,
  title    = {Enforcing conservation laws and dissipation inequalities numerically via auxiliary variables},
  volume   = {47},
  issn     = {1064-8275},
  doi      = {10.1137/25M1756673},
  number   = {6},
  journal  = {SIAM Journal on Scientific Computing},
  author   = {Andrews, B. D. and Farrell, P. E.},
  month    = dec,
  year     = {2025},
  pages    = {A3516--A3535}
}

@article{Bai2023,
  title     = {A new approach to the analysis of parametric finite element approximations to mean curvature flow},
  author    = {Bai, Genming and Li, Buyang},
  journal   = {Foundations of Computational Mathematics},
  volume    = {24},
  number    = {5},
  pages     = {1673--1737},
  year      = {2024},
  publisher = {Springer},
  doi       = {10.1007/s10208-023-09622-x}
}

@article{Bai2024,
  title   = {Convergence of a stabilized parametric finite element method of the {Barrett--Garcke--N{\"u}rnberg} type for curve shortening flow},
  author  = {Bai, Genming and Li, Buyang},
  journal = {Mathematics of Computation},
  year    = {2024},
  doi     = {10.1090/mcom/4019}
}

@article{Bao_Zhao_2021,
  title     = {A structure-preserving parametric finite element method for surface diffusion},
  author    = {Bao, Weizhu and Zhao, Quan},
  journal   = {SIAM Journal on Numerical Analysis},
  volume    = {59},
  number    = {5},
  pages     = {2775--2799},
  year      = {2021},
  publisher = {SIAM},
  doi       = {10.1137/21M1406751}
}

@article{Bao-Jiang-Li,
  title     = {A symmetrized parametric finite element method for anisotropic surface diffusion of closed curves},
  author    = {Bao, Weizhu and Jiang, Wei and Li, Yifei},
  journal   = {SIAM Journal on Numerical Analysis},
  volume    = {61},
  number    = {2},
  pages     = {617--641},
  year      = {2023},
  publisher = {SIAM},
  doi       = {10.1137/22M1472851}
}

@article{Bao-Li2023,
  title     = {A symmetrized parametric finite element method for anisotropic surface diffusion in three dimensions},
  author    = {Bao, Weizhu and Li, Yifei},
  journal   = {SIAM Journal on Scientific Computing},
  volume    = {45},
  number    = {4},
  pages     = {A1438--A1461},
  year      = {2023},
  publisher = {SIAM},
  doi       = {10.1137/22M1500575}
}

@article{Bao-Li2024,
  title     = {A structure-preserving parametric finite element method for geometric flows with anisotropic surface energy},
  author    = {Bao, Weizhu and Li, Yifei},
  journal   = {Numerische Mathematik},
  volume    = {156},
  number    = {2},
  pages     = {609--639},
  year      = {2024},
  publisher = {Springer},
  doi       = {10.1007/s00211-024-01398-8}
}

@article{Barrett_Garcke_Nurnberg_2007,
  author  = {Barrett, J. W. and  Garcke, H. and  N\"{u}rnberg, R.},
  journal = {Journal of Computational Physics},
  pages   = {441--467},
  title   = {A Parametric Finite Element Method for Fourth Order Geometric Evolution Equations},
  volume  = {222},
  year    = {2007},
  doi     = {10.1016/j.jcp.2006.07.026},
}

@article{Barrett_Garcke_Nurnberg_2008a,
  author  = {Barrett, J. W. and  Garcke, H. and  N\"{u}rnberg, R.},
  title   = {On the Parametric Finite Element Approximation of Evolving Hypersurfaces in $\mathbb{R}^3$},
  journal = {Journal of Computational Physics},
  volume  = {227},
  number  = {10},
  pages   = {4281--4307},
  year    = {2008},
  doi     = {10.1016/j.jcp.2007.11.023},
}

@article{Barrett_Garcke_Nurnberg_2008b,
  author  = {Barrett, J. W. and  Garcke, H. and  N\"{u}rnberg, R.},
  title   = {Parametric approximation of {W}illmore flow and related geometric evolution equations},
  journal = {SIAM Journal on Scientific Computing},
  volume  = {31},
  pages   = {225--253},
  year    = {2008},
  doi     = {10.1137/070700231},
}

@article{Barrett2006,
  title   = {Finite element approximation of a phase field model for surface diffusion of voids in a stressed solid},
  author  = {Barrett, John and Garcke, Harald and N{\"u}rnberg, Robert},
  journal = {Mathematics of Computation},
  volume  = {75},
  number  = {253},
  pages   = {7--41},
  year    = {2006},
  doi     = {10.1090/S0025-5718-05-01802-8}
}

@article{barrett2011approximation,
  title     = {The approximation of planar curve evolutions by stable fully implicit finite element schemes that equidistribute},
  author    = {Barrett, John W and Garcke, Harald and N{\"u}rnberg, Robert},
  journal   = {Numerical Methods for Partial Differential Equations},
  volume    = {27},
  number    = {1},
  pages     = {1--30},
  year      = {2011},
  publisher = {Wiley Online Library},
  doi       = {10.1002/num.20637},
}

@article{Betsch_Steinmann_2000a,
  title   = {Inherently energy conserving time finite elements for classical mechanics},
  volume  = {160},
  issn    = {0021-9991},
  doi     = {10.1006/jcph.2000.6427},
  number  = {1},
  journal = {Journal of Computational Physics},
  author  = {Betsch, P. and Steinmann, P.},
  month   = may,
  year    = {2000},
  pages   = {88--116}
}

@article{Betsch_Steinmann_2000b,
  title   = {Conservation properties of a time {FE} method—part {I}: time-stepping schemes for {N}-body problems},
  volume  = {49},
  issn    = {1097-0207},
  number  = {5},
  journal = {International Journal for Numerical Methods in Engineering},
  author  = {Betsch, P. and Steinmann, P.},
  year    = {2000},
  pages   = {599--638},
  doi     = {10.1002/1097-0207(20001020)49:5%3C599::AID-NME960%3E3.0.CO;2-9}
}

@incollection{BGN20,
  title     = {Parametric finite element approximations of curvature-driven interface evolutions},
  author    = {Barrett, John W. and Garcke, Harald and N{\"u}rnberg, Robert},
  booktitle = {Handbook of numerical analysis},
  volume    = {21},
  pages     = {275--423},
  year      = {2020},
  publisher = {Elsevier},
  doi       = {10.1016/bs.hna.2019.05.002}
}

@article{Bronsard1991,
  title     = {Motion by mean curvature as the singular limit of {Ginzburg--Landau} dynamics},
  author    = {Bronsard, Lia and Kohn, Robert V},
  journal   = {Journal of Differential Equations},
  volume    = {90},
  number    = {2},
  pages     = {211--237},
  year      = {1991},
  publisher = {Academic Press},
  doi       = {10.1016/0022-0396(91)90147-2}
}

@article{Cohen_Hairer_2011,
  title    = {Linear energy-preserving integrators for {Poisson} systems},
  volume   = {51},
  issn     = {1572-9125},
  doi      = {10.1007/s10543-011-0310-z},
  number   = {1},
  journal  = {BIT Numerical Mathematics},
  author   = {Cohen, D. and Hairer, E.},
  month    = mar,
  year     = {2011},
  pages    = {91--101}
}

@article{deckelnick2026second,
  title     = {Second order in time finite element schemes for curve shortening flow and curve diffusion},
  author    = {Deckelnick, Klaus and N{\"u}rnberg, Robert},
  journal   = {SIAM Journal on Numerical Analysis},
  volume    = {64},
  number    = {1},
  pages     = {103--124},
  year      = {2026},
  publisher = {SIAM},
  doi       = {10.1137/25M1737523}
}

@inproceedings{Desbrun1999,
  title     = {Implicit fairing of irregular meshes using diffusion and curvature flow},
  author    = {Desbrun, Mathieu and Meyer, Mark and Schr{\"o}der, Peter and Barr, Alan H},
  booktitle = {Proceedings of the 26th annual conference on Computer graphics and interactive techniques},
  pages     = {317--324},
  year      = {1999},
  doi       = {10.1145/311535.311576}
}

@article{Duan2021,
  title   = {High-order fully discrete energy diminishing evolving surface finite element methods for a class of geometric curvature flows},
  author  = {Duan, Beiping and Li, Buyang and Zhang, Zhiming},
  journal = {Annals of Applied Mathematics},
  volume  = {37},
  number  = {4},
  pages   = {405-436},
  year    = {2021},
  doi     = {10.4208/aam.OA-2021-0007},
}

@article{Duan2024,
  title     = {New artificial tangential motions for parametric finite element approximation of surface evolution},
  author    = {Duan, Beiping and Li, Buyang},
  journal   = {SIAM Journal on Scientific Computing},
  volume    = {46},
  number    = {1},
  pages     = {A587--A608},
  year      = {2024},
  publisher = {SIAM},
  doi       = {10.1137/23M1551857}
}

@misc{dune-meshdist,
  author       = {{DUNE Project}},
  title        = {{\texttt{dune-meshdist}}: a {DUNE} extension module for computing distances between triangulated surfaces},
  howpublished = {\url{https://gitlab.dune-project.org/extensions/dune-meshdist}},
  year         = {2024},
  note         = {Accessed: 2026-05-04}
}

@article{Dziuk_1994,
  author  = {Dziuk, D.},
  title   = {Convergence of a semi-discrete scheme for the curve shortening flow},
  journal = {Mathematical Models and Methods in Applied Sciences},
  volume  = {04},
  number  = {04},
  pages   = {589--606},
  year    = {1994},
  doi     = {10.1142/S0218202594000339}
}

@article{Egger_Habrich_Shashkov_2021,
  title    = {On the energy stable approximation of {Hamiltonian} and gradient systems},
  volume   = {21},
  issn     = {1609-9389},
  doi      = {10.1515/cmam-2020-0025},
  number   = {2},
  journal  = {Computational Methods in Applied Mathematics},
  author   = {Egger, H. and Habrich, O. and Shashkov, V.},
  month    = apr,
  year     = {2021},
  pages    = {335--349}
}

@article{farrell2020b,
  author  = {P. E. Farrell and R. C. Kirby and J. Marchena-Menendez},
  title   = {Irksome: automating {Runge--Kutta} time-stepping for finite element methods},
  year    = {2021},
  journal = {ACM Transactions on Mathematical Software},
  volume  = {47},
  issue   = {4},
  doi     = {10.1145/3466168}
}

@article{French_Schaeffer_1990,
  title   = {Continuous finite element methods which preserve energy properties for nonlinear problems},
  volume  = {39},
  issn    = {0096-3003},
  doi     = {10.1016/S0096-3003(20)80006-X},
  number  = {3},
  journal = {Applied Mathematics and Computation},
  author  = {French, D. A. and Schaeffer, J. W.},
  month   = oct,
  year    = {1990},
  pages   = {271--295}
}

@misc{Gao_Li_Tang_2026,
  title     = {Dual formulations of geometric curvature flows and their discretizations},
  doi       = {10.48550/arXiv.2604.18288},
  publisher = {arXiv},
  author    = {Gao, G. and Li, B. and Tang, R.},
  month     = apr,
  year      = {2026},
  note      = {arXiv manuscript}
}

@article{Gao-Garcke-Li-Tang,
  author  = {Gao, Guangwei and Garcke, Harald and Li, Buyang and Tang, Rong},
  title   = {An Energy-Stable Minimal Deformation Rate Scheme for Mean Curvature Flow and Surface Diffusion},
  journal = {SIAM Journal on Scientific Computing},
  volume  = {48},
  number  = {1},
  pages   = {A103-A131},
  year    = {2026},
  doi     = {10.1137/25M1753838},
}

@article{Gao-Li,
  title    = {Geometric-structure preserving methods for surface evolution in curvature flows with minimal deformation formulations},
  journal  = {Journal of Computational Physics},
  volume   = {524},
  pages    = {113718},
  year     = {2025},
  issn     = {0021-9991},
  doi      = {10.1016/j.jcp.2025.113718},
  author   = {Guangwei Gao and Buyang Li}
}

@article{Garcke_et_al_2025a,
  author  = {Garcke, H. and Jiang, W. and Su, C. and Zhang, G.},
  title   = {Structure-Preserving Parametric Finite Element Method for Surface Diffusion Based on {Lagrange} Multiplier Approaches},
  journal = {SIAM Journal on Scientific Computing},
  volume  = {47},
  number  = {3},
  pages   = {A1983-A2011},
  year    = {2025},
  doi     = {10.1137/24M1687546}
}

@article{Garcke_et_al_2025b,
  author  = {Garcke, H. and Nürnberg, R. and Praetorius, S. and Zhang, G.},
  title   = {Isoparametric finite element methods for mean curvature flow and surface diffusion},
  journal = {Journal of Computational Physics},
  volume  = {539},
  pages   = {114248},
  year    = {2025},
  doi     = {10.1016/j.jcp.2025.114248}
}

@article{Gilmer1972,
  title     = {Simulation of crystal growth with surface diffusion},
  author    = {Gilmer, GH and Bennema, P},
  journal   = {Journal of Applied Physics},
  volume    = {43},
  number    = {4},
  pages     = {1347--1360},
  year      = {1972},
  publisher = {American Institute of Physics},
  doi       = {10.1063/1.1661325},
}

@article{Gomer1990,
  title   = {Diffusion of adsorbates on metal surfaces},
  author  = {Gomer, Rep},
  journal = {Reports on Progress in Physics},
  volume  = {53},
  number  = {7},
  pages   = {917--1002},
  year    = {1990},
  doi     = {10.1088/0034-4885/53/7/002}
}

@article{Hairer_Lubich_2014,
  title    = {Energy-diminishing integration of gradient systems},
  volume   = {34},
  issn     = {0272-4979},
  doi      = {10.1093/imanum/drt031},
  number   = {2},
  journal  = {IMA Journal of Numerical Analysis},
  author   = {Hairer, E. and Lubich, C.},
  month    = apr,
  year     = {2014},
  pages    = {452--461}
}

@book{Hairer_Lubich_Wanner_2006,
  title      = {Geometric Numerical Integration: Structure-Preserving Algorithms for Ordinary Differential Equations},
  isbn       = {978-3-540-30666-5},
  shorttitle = {Geometric Numerical Integration},
  publisher  = {Springer Science \& Business Media},
  author     = {Hairer, E. and Lubich, C. and Wanner, G.},
  month      = may,
  year       = {2006},
  address    = {Heidelberg, Germany}
}

@manual{ham2023c,
  title        = {Firedrake User Manual},
  author       = {D. A. Ham and P. H. J. Kelly and L. Mitchell and C. J. Cotter and R. C. Kirby and K. Sagiyama and N. Bouziani and S. Vorderwuelbecke and T. J. Gregory and J. Betteridge and D. R. Shapero and R. W. Nixon-Hill and C. J. Ward and P. E. Farrell and P. D. Brubeck and I. Marsden and T. H. Gibson and M. Homolya and T. Sun and A. T. T. McRae and F. Luporini and A. Gregory and M. Lange and S. W. Funke and F. Rathgeber and G.-T. Bercea and G. R. Markall},
  organization = {Imperial College London and University of Oxford and Baylor University and University of Washington},
  year         = {2023},
  doi          = {10.25561/104839}
}

@article{Hu_Li_2022,
  author  = {Hu, J. and Li, B.},
  journal = {Numerische Mathematik},
  number  = {1},
  title   = {Evolving finite element methods with an artificial tangential velocity for mean curvature flow and {Willmore} flow},
  volume  = {152},
  year    = {2022},
  pages   = {127--181},
  doi     = {10.1007/s00211-022-01309-9},
}

@article{jiang2024stable,
  title     = {Stable backward differentiation formula time discretization of {BGN}-based parametric finite element methods for geometric flows},
  author    = {Jiang, Wei and Su, Chunmei and Zhang, Ganghui},
  journal   = {SIAM Journal on Scientific Computing},
  volume    = {46},
  number    = {5},
  pages     = {A2874--A2898},
  year      = {2024},
  publisher = {SIAM},
  doi       = {10.1137/23M1625597},
}

@article{jiang2025predictor,
  title     = {Predictor-corrector, {BGN}-based parametric finite element methods for surface diffusion},
  author    = {Jiang, Wei and Su, Chunmei and Zhang, Ganghui and Zhang, Lian},
  journal   = {Journal of Computational Physics},
  volume    = {530},
  pages     = {113901},
  year      = {2025},
  publisher = {Elsevier},
  doi       = {10.1016/j.jcp.2025.113901}
}

@article{Jiang21,
  title     = {A perimeter-decreasing and area-conserving algorithm for surface diffusion flow of curves},
  author    = {Jiang, Wei and Li, Buyang},
  journal   = {Journal of Computational Physics},
  volume    = {443},
  pages     = {110531},
  year      = {2021},
  publisher = {Elsevier},
  doi       = {10.1016/j.jcp.2021.110531}
}

@article{Jiang23,
  title     = {A second-order in time, {BGN}-based parametric finite element method for geometric flows of curves},
  author    = {Jiang, Wei and Su, Chunmei and Zhang, Ganghui},
  journal   = {Journal of Computational Physics},
  volume    = {514},
  pages     = {113220},
  year      = {2024},
  publisher = {Elsevier},
  doi       = {10.1016/j.jcp.2024.113220}
}

@article{Kovacs-Li-Lubich2019,
  title     = {A convergent evolving finite element algorithm for mean curvature flow of closed surfaces},
  author    = {Kov{\'a}cs, Bal{\'a}zs and Li, Buyang and Lubich, Christian},
  journal   = {Numerische Mathematik},
  volume    = {143},
  pages     = {797--853},
  year      = {2019},
  publisher = {Springer},
  doi       = {10.1007/s00211-019-01074-2}
}

@article{Li1,
  title     = {Convergence of {Dziuk}'s linearly implicit parametric finite element method for curve shortening flow},
  author    = {Li, Buyang},
  journal   = {SIAM Journal on Numerical Analysis},
  volume    = {58},
  number    = {4},
  pages     = {2315--2333},
  year      = {2020},
  publisher = {SIAM},
  doi       = {10.1137/19M1305483}
}

@article{Li1999,
  title     = {A numerical study of electro-migration voiding by evolving level set functions on a fixed {Cartesian} grid},
  author    = {Li, Zhilin and Zhao, Hongkai and Gao, Huajian},
  journal   = {Journal of Computational Physics},
  volume    = {152},
  number    = {1},
  pages     = {281--304},
  year      = {1999},
  publisher = {Elsevier},
  doi       = {10.1006/jcph.1999.6249}
}

@article{Li2,
  title     = {Convergence of {Dziuk}'s semidiscrete finite element method for mean curvature flow of closed surfaces with high-order finite elements},
  author    = {Li, Buyang},
  journal   = {SIAM Journal on Numerical Analysis},
  volume    = {59},
  number    = {3},
  pages     = {1592--1617},
  year      = {2021},
  publisher = {SIAM},
  doi       = {10.1137/20M136935X}
}

@article{McLachlan_Quispel_Robidoux_1999,
  title    = {Geometric integration using discrete gradients},
  volume   = {357},
  doi      = {10.1098/rsta.1999.0363},
  number   = {1754},
  journal  = {Philosophical Transactions of the Royal Society of London. Series A: Mathematical, Physical and Engineering Sciences},
  author   = {McLachlan, R. I. and Quispel, G. R. W. and Robidoux, N.},
  month    = apr,
  year     = {1999},
  pages    = {1021--1045}
}

@article{Mullins1956,
  title     = {Two-dimensional motion of idealized grain boundaries},
  author    = {Mullins, William W},
  journal   = {Journal of Applied Physics},
  volume    = {27},
  number    = {8},
  pages     = {900--904},
  year      = {1956},
  publisher = {American Institute of Physics},
  doi       = {10.1063/1.1722511},
}

@article{Mullins1957,
  title     = {Theory of thermal grooving},
  author    = {Mullins, William W},
  journal   = {Journal of Applied Physics},
  volume    = {28},
  number    = {3},
  pages     = {333--339},
  year      = {1957},
  publisher = {American Institute of Physics},
  doi       = {10.1063/1.1722742},
}

@article{Osher1988,
  title     = {Fronts propagating with curvature-dependent speed: Algorithms based on {Hamilton--Jacobi} formulations},
  author    = {Osher, Stanley and Sethian, James A},
  journal   = {Journal of Computational Physics},
  volume    = {79},
  number    = {1},
  pages     = {12--49},
  year      = {1988},
  publisher = {Elsevier},
  doi       = {10.1016/0021-9991(88)90002-2}
}

@article{rathgeber2016,
  doi     = {10.1145/2998441},
  year    = 2016,
  volume  = {43},
  number  = {3},
  pages   = {1--27},
  author  = {F. Rathgeber and D. A. Ham and L. Mitchell and M. Lange and F. Luporini and A. T. T. Mcrae and G.-T. Bercea and G. R. Markall and P. H. J. Kelly},
  title   = {Firedrake: automating the finite element method by composing abstractions},
  journal = {{ACM} Transactions on Mathematical Software}
}

@article{Wang2015,
  title     = {Sharp interface model for solid-state dewetting problems with weakly anisotropic surface energies},
  author    = {Wang, Yan and Jiang, Wei and Bao, Weizhu and Srolovitz, David J},
  journal   = {Physical Review B},
  volume    = {91},
  number    = {4},
  pages     = {045303},
  year      = {2015},
  publisher = {APS},
  doi       = {10.1103/PhysRevB.91.045303},
}

@article{Zhao2021,
  title     = {An energy-stable parametric finite element method for simulating solid-state dewetting},
  author    = {Zhao, Quan and Jiang, Wei and Bao, Weizhu},
  journal   = {IMA Journal of Numerical Analysis},
  volume    = {41},
  number    = {3},
  pages     = {2026--2055},
  year      = {2021},
  publisher = {Oxford University Press},
  doi       = {10.1093/imanum/draa070}
}

@misc{codes_repository,
  title={Codes: Arbitrary order structure-preserving discretizations for geometric curvature flows},
  author={Zhang, G and Andrews, Boris D. and Farrell, Patrick E.},
  howpublished={\url{https://github.com/Ganghui-Zhang/Arbitrary-order-structure-preserving-discretizations-for-geometric-curvature-flows}},
  year={2026},
}

@article{Akrivis_Li_Tang_Zhang_2025,
  title={High-order mass-, energy- and momentum-conserving methods for the nonlinear {S}chr{\"o}dinger equation},
  author={Akrivis, Georgios and Li, Buyang and Tang, Rong and Zhang, Hui},
  journal={Journal of Computational Physics},
  volume={532},
  pages={113974},
  year={2025},
  publisher={Elsevier},
  doi   = {10.1016/j.jcp.2025.113974}
}

\end{document}